\documentclass{article}

\usepackage{amssymb}
\usepackage{amsbsy}
\usepackage{amsmath}
\usepackage{amsfonts}
\usepackage{amscd}
\usepackage{mathrsfs}
\usepackage{amsthm}

\usepackage{cite}
\usepackage{graphicx}
\usepackage{atbegshi}% http://ctan.org/pkg/atbegshi
\AtBeginDocument{\AtBeginShipoutNext{\AtBeginShipoutDiscard}}
\usepackage[all]{xy}

\usepackage{subfigure}
\usepackage[left=2cm,right=2cm,top=2cm,bottom=2cm]{geometry}

\def\ad{\operatorname{ad}}

\theoremstyle{plain}
\newtheorem{theorem}{Theorem}
\newtheorem{corollary}[theorem]{Corollary}
\newtheorem{proposition}[theorem]{Proposition}
\newtheorem{lemma}[theorem]{Lemma}
 \theoremstyle{definition}

 \theoremstyle{remark}

\DeclareMathOperator{\Exp}{Exp}

\DeclareMathOperator{\tr}{tr}
\DeclareMathOperator{\spann}{span}
\DeclareMathOperator{\rank}{rank}

\DeclareMathOperator{\Id}{Id}
\DeclareMathOperator{\I}{I}
\DeclareMathOperator{\SO}{SO}

\DeclareMathOperator{\so}{so}
\def\sn{\operatorname{sn}}
\def\cn{\operatorname{cn}}
\def\dn{\operatorname{dn}}
\def\am{\operatorname{am}}
\def\Kill{\operatorname{Kill}}
\def\spann{\operatorname{span}}
\def\sign{\operatorname{sign}}

\begin{document}
\title{Sub-Riemannian and almost-Riemannian \\ geodesics on $\SO(3)$ and $S^2$}

\author{I. Yu. Beschastnyi\thanks{Program Systems Institute, Pereslavl-Zalessky, Russia, i.beschastnyi@gmail.com} \and Yu. L. Sachkov\thanks{Program Systems Institute, Pereslavl-Zalessky, Russia, yusachkov@gmail.com}}

\thanks{Work supported by 
Grant of the Russian Federation for the State Support of Researches
(Agreement  No14.B25.31.0029).}

%\subjclass{49J15, 93C10, 53C17, 22E30}

%\keywords{sub-Riemannian geometry, almost-Riemannian geometry, optimal control, Maxwell time, cut time, exponential mapping}

\maketitle

\begin{abstract}
In this paper we study geodesics of left-invariant sub-Rie\-man\-nian metrics on SO(3) and almost-Rie\-man\-nian metrics on $S^2$. These structures are connected with each other, and it is possible to use information about one of them to obtain results about another one. We give an explicit parameterization of sub-Riemannian geodesics on SO(3) and use it to get a parameterization of almost-Riemannian geodesics on $S^2$. We use symmetries of the exponential map to obtain some necessary optimality conditions. We present some upper bounds on the cut time in both cases and describe periodic geodesics on SO(3).
\end{abstract}

\section*{Introduction}
\label{sec:intro}

A sub-Riemannian manifold is a triple $(M,\Delta,g)$, where $M$ is a smooth connected manifold, $\Delta$ is a smooth constant rank distribution on $M$ and $g$ is a smooth Riemannian metric on $\Delta$. Sub-Riemannian structures often appear in applications, like quantum control~\cite{boscain_quant}, robotics~\cite{sachkov_se2,sachkov2,sachkov3}, image manipulation~\cite{bdrs} and many others~\cite{montgomery}.

In a recent article~\cite{ab_classification}, a full classification of left-invariant sub-Riemannian structures on 3D Lie groups was given. These structures give the basic and most simple examples of sub-Riemannian manifolds. That is why they are often used as models for general techniques and as a source of new ideas and intuition for studying more complex sub-Riemannian manifolds. 

One of the most important issues in Riemannian geometry and its generalizations is the description of the minimal (shortest) curves. This problem can be rather hard, and even in the simplest case of left-invariant 3D sub-Riemannian manifolds a full description of minimal curves is known only in a small number of cases: the sub-Riemannian problem on the Heisenberg group~\cite{montgomery,vershik}, its spherical and hyperbolic analogs~\cite{boscain_rossi,berestovsky,markina1,markina2} and the sub-Riemannian problem on SE(2)~\cite{berestovsky_se2,sachkov_se2,sachkov2,sachkov3}. Some significant progress was also made in the case of SH(2)~\cite{butt}. 

In general it is known that geodesic flows of sub-Riemannian structures on 3D unimodular Lie groups are Liouville integrable~\cite{integrability1,integrability2}. Using a notion of curvature for sub-Riemannian manifolds, authors of paper~\cite{curvature} were able to provide estimates on the conjugate time for the same groups. But a characterization of shortest left-invariant sub-Riemannian geodesics on SO(3) and SL(2) is still unknown.

Geodesics of sub-Riemannian metrics and Riemannian metrics on SO(3) behave similarly in many ways. The main reason for this is that any contact sub-Riemannian metric can be obtained as a limit of some family of Riemannian metrics (the penalty metric, see~\cite{montgomery}). After a suitable change of coordinates a Riemannian metric $g$ becomes diagonal:
$$
g_{ij}=\begin{pmatrix}
I_1 & 0 & 0 \\
0 & I_2 & 0 \\
0 & 0 & I_3
\end{pmatrix}.
$$  
It is well known that Riemannian geodesics on SO(3) describe motions of a free rigid body. The constants $I_j$ depend on the  mass distribution of the body and are called the principal inertia moments. A study of the rotational movement of rigid bodies was initiated by Euler. In 1766 he wrote down and integrated equations of motion in the (Euler) case $I_1 = I_2 = I_3$~\cite{euler}. Later, in 1788, Lagrange obtained a parameterization of geodesics for the (Lagrange) case, when just two principal inertia moments are equal one to another~\cite{lagrange}. In the general case equations of motion were integrated by Jacobi in 1849 after he introduced his famous elliptic functions~\cite{jacobi}. Nowadays the free rigid body dynamics became a classic topic in mechanics. It is discussed in a number of different text books, like~\cite{landau} or~\cite{whit}. Nevertheless, the optimality question seems still to be open in the general case (see~\cite{berger}, Section 6.5.4). 

One obtains a sub-Riemannian structure by passing to a limit $I_j \rightarrow +\infty$ for some $j$. There is no physical rigid body that corresponds to the sub-Riemannian case, because there are some additional physical constraints on the inertia moments~\cite{whit}, the triangle inequalities:
$$
I_1+I_2\leq I_3, \qquad
I_2+I_3\leq I_1, \qquad
I_3+I_1\leq I_2.
$$
Nevertheless the sub-Riemannian geodesics still have meaningful applications~\cite{car}. Our initial goal was to obtain a full description of minimal curves on SO(3) equipped with a one-parametric family of left-invariant sub-Riemannian metrics. In this family there is one particular symmetric structure that corresponds to a bi-invariant metric, which was completely examined earlier in papers~\cite{boscain_rossi,berestovsky,markina1,markina2}. Thus in this article we consider left-invariant metrics that are not bi-invariant. Although we have not obtained a full description of the minimal curves, we give new necessary optimality conditions for geodesic curves and some new properties of periodic geodesics. A very brief description of periodic geodesics was previously given in~\cite{vershik}. In this paper we investigate their topological properties and give specific conditions for a geodesic to be periodic.

In the second part of the article we study almost-Riemannian problems on the two-sphere. Naively an almost-Riemannian manifold is obtained in the following way: take an $n$-dimensional smooth manifold, $n$ vector fields that are linearly independent almost everywhere and define a metric in which these vector fields form an orthonormal frame. The set of points where these $n$ vector fields are linearly dependent is called the singular set. Given a sub-Riemannian structure on SO(3), one can project it down to its homogeneous space $S^2$. After this procedure the two-sphere is endowed with a structure of an almost-Riemannian manifold. 

Almost-Riemannian structures arise in problems of population transfer in quantum mechanics~\cite{boscain_quant,boscain_quant2} and in the problem of orbital transfer in space mechanics~\cite{space}. Geodesics on almost-Riemannian two-spheres were previously studied in a series of papers~\cite{boscain_quant,boscain_quant2} in a context of quantum control. Remarkably, authors of these articles were able to provide an optimal synthesis for a particular initial point on $S^2$ without a full parameterization of geodesics. In this paper we study the symmetries of the exponential mapping in almost-Riemannian problems on $S^2$ and obtain some necessary optimality conditions. We then use them to obtain some new bounds on the cut time for almost-Riemannian geodesics on $S^2$ and sub-Riemannian geodesics on SO(3).

We also note, that during preparation of this manuscript article~\cite{bonnard1} and preprint~\cite{bonnard2} appeared, where the same sub-Riemannian and almost-Riemannian problems were studied. Thus it is  reasonable to indicate explicitly the novelty of some results in this paper. We integrate the Hamiltonian system for sub-Riemannian geodesics on SO(3) using a well-known approach from mechanics~\cite{landau,jurdjevic}, and action-angle coordinates in the dual of so(3) induced by a pendulum~\cite{sachkov_se2}. In~\cite{bonnard1} and~\cite{bonnard2} the authors gave a similar parameterization using the same technique, but omitted details. Here we give a full derivation for parameterization of sub-Riemannian geodesics and use it to obtain a parameterization for almost-Riemannian geodesics. Using these formulas we give a novel characterization of periodic geodesics on SO(3) and study their topological properties. The description of symmetries of the exponential map and necessary optimality conditions in the sub-Riemannian problem are essentially new.

In~\cite{bonnard2} necessary and sufficient optimality conditions are formulated for geodesics which start from the singular set. In this paper we show that the sub-Riemannian structure on SO(3) and almost-Riemannian structure on $S^2$ share a number of discrete symmetries. We use this fact to obtain some optimality conditions in the almost-Riemannian case. This technique allows us only to give necessary conditions, but we state them for some initial points that lie outside the singular set. We use the parameterization of geodesics from this paper and results from~\cite{bonnard2} to obtain bounds on the cut time for almost-Riemannian geodesics that start from the singular set and for any sub-Riemannian geodesics on SO(3). Thus the second part of this paper may be considered as  complementary to~\cite{bonnard1} and~\cite{bonnard2}, where authors have proved many interesting results.

In the following text we use the following notations:
\begin{itemize}
\item $A_i$ is the standard basis of $\so(3)=T_{\Id}\SO(3)$
\begin{equation}
A_1 = \begin{pmatrix}
0 & 0 & 0 \\
0 & 0 & -1 \\
0 & 1 & 0
\end{pmatrix},
%----------------
\quad
%----------------
A_2 = \begin{pmatrix}
0 & 0 & 1 \\
0 & 0 & 0 \\
-1 & 0 & 0
\end{pmatrix},
%----------------
\quad
%----------------
A_3 = \begin{pmatrix}
0 & -1 & 0 \\
1 & 0 & 0 \\
0 & 0 & 0
\end{pmatrix};
\end{equation}
\item $e_i$ is the basis of $\mathbb{R}^3$
\begin{equation}
e_1 = \begin{pmatrix}
1 \\
0 \\
0
\end{pmatrix},
%----------------
\quad
%----------------
e_2 = \begin{pmatrix}
0 \\
1\\
0
\end{pmatrix},
%----------------
\quad
%----------------
e_3 = \begin{pmatrix}
0 \\
0 \\
1
\end{pmatrix};
\end{equation}
\item $i,j,k$ is the basis in the space $\I$ of imaginary quaternions;
\item $\eta^l$ is the basis in $T_{\Id}^*\SO(3)$ dual to $A_l$, i.e., $\langle\eta^l,A_m\rangle=\delta^l_m$, $l,m=1,2,3$.
\end{itemize}

By a capital letter we denote an element of $\so(3)$, by a small letter --- the corresponding imaginary quaternion, and by a small letter with an arrow --- the corresponding vector in $\mathbb{R}^3$:
$$
\Omega \simeq \omega\simeq \vec{\omega},\qquad \Omega\in\so(3), \qquad \omega\in\I, \qquad \vec{\omega}\in\mathbb{R}^3.
$$

The structure of this paper is as follows. In Section~\ref{sec:SR-problem} we formulate the sub-Riemannian problem. In Section~\ref{sec:SR-geodesics} we write down the Hamiltonian system of the Pontryagin maximum principle and integrate it. Periodic geodesics are studied in Section~\ref{sec:periodic}. Symmetries and necessary optimality conditions are given in Section~\ref{sec:symmetries}. In Section~\ref{sec:AR-problem} almost-Riemannian structures on $S^2$ are defined and the connection with sub-Riemannian structures on SO(3) is explained. Symmetries and bounds on the cut time in the family of almost-Riemannian problems are given in Section~\ref{sec:AR-symmetries}. 

For the reader's convenience we have summarized all necessary definitions and properties of the elliptic integrals and Jacobi elliptic functions in Appendix B. The isomorphism between the space $\I$ of imaginary quaternions, $\mathbb{R}^3$ and $\so(3)$ is defined by (\ref{algebra_iso}) in Appendix A, where the necessary information about the space of quaternions is collected. 

\section{Statement of sub-Riemannian problem on SO(3)}
\label{sec:SR-problem}

Consider the Lie group SO(3) of rotations of the 3-dimensional space. We can define a left-invariant distribution in two equivalent ways: as a kernel of a left-invariant one-form or as a linear span of two linearly independent left-invariant vector fields $RX_1, RX_2$. Here $X_1,X_2$ are elements of the Lie algebra so(3) and $R\in\SO(3)$. If the distribution $\Delta$ is contact, then $[X_1,X_2]\notin \Delta$. One can define a left-invariant metric $g$ on $\Delta$ by declaring $X_1$, $X_2$ orthonormal for $g$.

From the classification of left-invariant structures on 3D Lie groups~\cite{ab_classification} it follows that $X_1$ and $X_2$ can be chosen to satisfy the following structure equations:
\begin{equation}
\label{structure_equations}
[X_2,X_1]=X_3,\qquad [X_1,X_3]=(\kappa + \chi) X_2,\qquad [X_2,X_3]=(\chi - \kappa) X_1,
\end{equation}
where $\kappa\in\mathbb{R}$, $\chi\geq 0$ are two differential invariants of the sub-Riemannian structure that satisfy $\kappa - \chi \geq 0$ in the case of SO(3). The scaling of the frame $\{X_1,X_2\}$ scales proportionally the distance function and both invariants $\kappa$ and $\chi$. Thus authors of~\cite{ab_classification} considered normalized structures for which $\kappa^2 +\chi^2 = 1$. For further calculations in this paper it is more suitable to use the normalization $\kappa + \chi = 1$. Let also $a\in [0,1)$ be the invariant defined by $a = \sqrt{2\chi}$. Then all non-isometric sub-Riemannian structures on SO(3) are parameterized by $a$.

It is easily verified that the Lie algebra elements
\begin{equation}
X_1 = A_2, \qquad X_2 = \sqrt{1-a^2}A_1, \qquad X_3 = \sqrt{1-a^2}A_3
\end{equation}
satisfy the above structure equations.

A Lipschitz continuous curve $R:[0,T]\to\SO(3)$ is called horizontal if for a.e. $t\in [0,T]$ we have $\dot{R}(t)\in\Delta_{R(t)}$. The length of a horizontal curve is defined as usual:
$$
l(R)=\int_0^T\sqrt{g(\dot{R}(t),\dot{R}(t))}dt.
$$

Our goal is to find minimal horizontal curves that connect two given points $R_0,R_1\in \SO(3)$. Since the problem is left-invariant, we can assume that $R_0$ is the identity element of SO(3). By the Cauchy-Shwartz inequality, minimization of the sub-Riemannian length is equivalent to minimization of the action functional
$$
\frac{1}{2}\int_0^T g(\dot{R}(t),\dot{R}(t))dt \rightarrow \min
$$
with fixed $T$. Thus we can formulate the problem of finding minimal curves as an optimal control problem of the form:
\begin{equation}
\label{R_system}
\dot{R}=R(u_1X_1+u_2X_2), \qquad R\in \SO(3), \qquad (u_1,u_2)\in\mathbb{R}^2,
\end{equation}
\begin{equation}
\label{b_conditions}
R(0)= \Id, \qquad R(T) = R_1,
\end{equation}
\begin{equation}
\label{sr_action}
\int_0^T\frac{u_1^2 + u_2^2}{2} dt \rightarrow \min, \qquad T \text{ is fixed.}
\end{equation}

Since $\Delta_R + [\Delta,\Delta]_R=T_R\SO(3)$, the system has full rank and is thus completely controllable. We can reduce the given optimal control problem with a quadratic cost to a time optimal control problem with the same dynamics (\ref{R_system}), the same boundary conditions (\ref{b_conditions}), but with constraints $u_1^2+u_2^2 \leq 1$ and time minimization functional $T\rightarrow \min$ (see, for example, \cite{sachkov_se2}). After that we can apply Filippov's Theorem to establish existence of minimizing curves~\cite{agrachev_sachkov}.  

If $a = 0$, then we get the Lagrange sub-Riemannian case, meaning that this sub-Riemannian metric is a limit of Lagrange Riemannian metrics with $I_1 = I_2 = 1$ and $I_3 \to +\infty$. This case was completely studied in~\cite{boscain_rossi}, where the cut time for each trajectory was found, and in~\cite{berestovsky}, where analytic expressions for the sub-Riemannian spheres were given. In this particular case we have an additional rotational symmetry and a nice geometric interpretation: the sub-Riemannian problem is just the isoperimetric problem on the sphere. In the rest of the article we assume that $a\in(0,1)$.

\section{Parameterization of sub-Riemannian geodesics}
\label{sec:SR-geodesics}

Next we apply the Pontryagin maximum principle (PMP) to obtain a parameterization of geodesics, i.e., curves whose short arcs are length minimizers. Let $\mathfrak{g}^*$ be the dual of Lie algebra $\mathfrak{g}=\so(3)$. We introduce the control-dependent Hamiltonian of PMP:
$$
H_u(p) = \langle p,u_1X_1+u_2X_2 \rangle +\frac{\nu}{2}(u_1^2+u_2^2), \qquad p\in\mathfrak{g}^*,\; \nu\leq 0.
$$

\begin{theorem}[Pontryagin maximum principle~\cite{pmp,agrachev_sachkov}] 
If a pair $(u(t),R_t)$ is optimal, $t\in[0,T]$, then there is a Lipschitz curve $p(t)\in \mathfrak{g}^*$ and $\nu\leq 0$ such that:
\begin{enumerate}
\item $(p(t),\nu)\neq 0$,
\item $\left\{\begin{matrix}
\dot{R} = R\dfrac{\partial H}{\partial p},\\ 
\dot{p}=\ad^*\dfrac{\partial H}{\partial p}p;
\end{matrix}\right.$
\item $H(p) = \max_{u\in\mathbb{R}^2}H_u(p).$
\end{enumerate}
\end{theorem}
Here $\ad^*(\cdot)$ is the coadjoint representation of the Lie algebra $\mathfrak{g}$.

Since in the contact case there are no non-constant abnormal geodesics~\cite{ab_classification}, we can assume that $\nu=-1$. The maximized Hamiltonian of PMP is
$$
H(p)=\max_{u\in \mathbb{R}^2}H_u(p) = \frac{\langle p,X_1\rangle^2 + \langle p,X_2 \rangle^2}{2}=\frac{p_1^2+p_2^2}{2}
$$
with controls
$$
u_1 = \langle p,X_1\rangle = p_1, \qquad  u_2 = \langle p,X_2\rangle = p_2.
$$
Thus we have
$$
\frac{\partial H}{\partial p} = p_1X_1 + p_2X_2.
$$

It is easy to see that
$$
\ad\frac{\partial H}{\partial p} = \begin{pmatrix}
0 & 0 & -p_2(1-a^2) \\
0 & 0 & p_1 \\
p_2 & -p_1 & 0 
\end{pmatrix}.
$$
Then
$$
\begin{pmatrix}
\dot{p}_1 & \dot{p}_2 & \dot{p}_3
\end{pmatrix} = 
\begin{pmatrix}
p_1 & p_2 & p_3
\end{pmatrix}
\begin{pmatrix}
0 & 0 & -p_2(1-a^2) \\
0 & 0 & p_1 \\
p_2 & -p_1 & 0 
\end{pmatrix}
$$
and we get the following expression for the Hamiltonian system:
\begin{equation}
\label{horizontal_subsystem}
\dot{R} = R(p_1X_1 + p_2X_2),
\end{equation}
\begin{align}
\dot{p}_1 &= p_3p_2, \nonumber\\
\dot{p}_2 &= -p_3p_1, \label{vertical_subsystem}\\
\dot{p}_3 &= a^2 p_1p_2.\nonumber
\end{align}

We will perform integration of the Hamiltonian system in three steps. First we integrate the vertical subsystem (\ref{vertical_subsystem}), since its right-hand side does not depend on the horizontal variables. After a simple change of variables the vertical subsystem is transformed into the equations of mathematical pendulum for which explicit solution is known. Next we rewrite the vertical subsystem in the so-called Lax form and use an Euler angles parameterization for matrix $R_t$ to obtain expressions for two of three Euler angles without solving the corresponding differential equations. Finally, we use all previous results to integrate the ODE for the last angle in terms of the elliptic integral of the third kind.

Now we begin the first step. Consider extremal curves parameterized by arclength. In this case we have $H=\frac{1}{2}$ and we can express $p_1$ and $p_2$ in the following way:
\begin{equation}
\label{zamena}
p_1 = \cos\psi, \qquad p_2 = -\sin\psi.
\end{equation}
The vertical system (\ref{vertical_subsystem}) becomes
\begin{equation}
\label{vertical}
\begin{array}{l}
\dot{\psi} = p_3, \\
\dot{p}_3 = -\dfrac{a^2}{2} \sin 2\psi
\end{array}
\end{equation}
where $\psi \in S=\mathbb{R}/2\pi\mathbb{Z}$.

The cylinder $H=\frac{1}{2}$ is divided into regions determined by the energy of the pendulum $E=2p_3^2 - a^2\cos2\psi$:
\begin{align*}
C_1 &= \{\lambda\in C: E\in(-a^2,a^2)\}, \\
C_2 &= \{\lambda\in C: E\in(a^2, +\infty)\}, \\
C_3 &= \{\lambda\in C: E = a^2, p_3 \neq 0\}, \\
C_4 &= \{\lambda\in C: E = -a^2\}, \\
C_5 &= \{\lambda\in C: E = a^2, p_3 = 0\}.
\end{align*} 

We use different coordinates for different regions~\cite{lawden} (for definitions of the Jacobi elliptic functions see Appendix B):
\begin{itemize}
\item Elliptic coordinates in $C_1$:
\begin{align*}
\sin\psi &= s_1k\sn\left(a\theta,k^2\right), & &s_1 = \sign\left( \cos\psi \right),\\
\cos\psi &= s_1\dn\left(a\theta,k^2\right), & &k\in(0,1),\\
p_3 &= ak\cn\left(a\theta,k^2\right), & &\theta\in[0,4K(k^2)/a].
\end{align*}
\item Elliptic coordinates in $C_2$:
\begin{align*}
\sin\psi &= s_2\sn\left(\frac{a\theta}{k},k^2\right), & &s_2 = \sign(c),\\
\cos\psi &= \cn\left(\frac{a\theta}{k},k^2\right), & &k\in (0,1),\\
p_3 &= \frac{s_2a}{k}\dn\left(\frac{a\theta}{k},k^2\right), & &\theta \in[0,4kK(k^2)/a].
\end{align*}
\item Elliptic coordinates in $C_3$:
\begin{align*}
\sin\psi &= s_1s_2\tanh a\theta,\\
\cos\psi &= \frac{s_1}{\cosh a\theta},\\
p_3 &= \frac{s_2a}{\cosh a\theta}, \\
\theta&\in(-\infty,+\infty).
\end{align*}
\item Solution in $C_4$:
\begin{align*}
\psi = \pi &n , n\in \mathbb{Z},\\
p_3 &= 0.
\end{align*}
\item Solution in $C_5$:
\begin{align*}
\psi = -\frac{\pi}{2}&+\pi n, n\in \mathbb{Z},\\
p_3 &= 0.
\end{align*}
\end{itemize}

For integration of the horizontal subsystem (\ref{horizontal_subsystem}) we follow a technique that is well known in the literature~\cite{landau,jurdjevic}. First we rewrite the vertical subsystem (\ref{vertical_subsystem}) in Lax form.

The Killing form $\Kill(\cdot,\cdot)$ allows us to identify a semisimple Lie algebra $\mathfrak{g}$ with its dual $\mathfrak{g}^*$ via the isomorphism:
$$
\Kill: X \mapsto \Kill(X,\cdot)\in\mathfrak{g}^*
$$ 
for any $X\in\mathfrak{g}$. The biinvariance condition for the Killing form  can be written in the following way:
$$
\Kill([Z, X], Y) + \Kill(X,[Z,Y]) = 0.
$$

Let $L\in \mathfrak{g}$ be a vector dual to $p\in\mathfrak{g}^*$. Take an arbitrary element $A\in \mathfrak{g}$. Using the equality $\frac{\partial H}{\partial p} = p_1X_1 +p_2X_2 = \Omega$, we obtain
\begin{align*}
\Kill(\dot{L},A) = \langle \dot{p},A\rangle = \langle \ad^*(\Omega)p,A\rangle =  \langle p,[\Omega, A]\rangle = \Kill( L,[\Omega, A]) = \Kill([L, \Omega], A).
\end{align*}
Since the Killing form is non-degenerate and $A\in\mathfrak{g}$ is arbitrary, we have
$$
\dot{p} =\ad^*\frac{\partial H}{\partial p}p \Leftrightarrow \dot{L}=[L,\Omega].
$$
This is the celebrated Lax equation~\cite{jurdjevic}.

The Killing form takes the simplest form in the basis $A_i$. If $X=x_1A_1+x_2A_2+x_3A_3$ and $Y=y_1A_1+y_2A_2+y_3A_3$ then 
$$
\Kill(X,Y)=-\frac{1}{2}\tr(XY)=\sum_{i=1}^3x_iy_i.
$$

Let $\nu^i$ be the basis in $\mathfrak{g}^*$ dual to $X_i$, i.e.,
 $$\langle\nu^i,X_j\rangle=\delta^i_j, \qquad i,j=1,2,3.$$ 
 The basis $\eta^i$ depends on $\nu^i$ in the following way:
\begin{eqnarray*}
\eta^2 = \nu^1; \\
\eta^1 = \sqrt{1-a^2}\nu^2; \\
\eta^3 = \sqrt{1-a^2}\nu^3.
\end{eqnarray*}
Consequently, the Lax equation takes the form:
$$
\frac{\dot{p}_2}{\sqrt{1-a^2}}A_1 + \dot{p}_1A_2 + \frac{\dot{p}_3}{\sqrt{1-a^2}}A_3 = \left[\frac{p_2}{\sqrt{1-a^2}}A_1 + p_1A_2 + \frac{p_3}{\sqrt{1-a^2}}A_3, p_2\sqrt{1-a^2}A_1 + p_1A_2 \right].
$$
If we take $P=\sqrt{1-a^2}L$, then $P$, $\Omega$ is still a Lax pair:
$$
\dot{P}=[P,\Omega].
$$

It is easy to check that
\begin{equation*}
P_t = R_t^{-1}P_0R_t
\end{equation*}
is in fact a solution of the Lax equation. Using the isomorphism between so(3) and $\mathbb{R}^3$ we can rewrite this solution in the form
\begin{equation}
\label{lax_solution}
\vec{p}_t = R_t^{-1}\vec{p}_0.
\end{equation}
From this it follows that the length of $\vec{p}_t$ is a conserved quantity. We denote its square by $M  = |\vec{p}_t|^2 = p_2^2 + p_1^2(1-a^2)+p_3^2$. 

In order to make use of (\ref{lax_solution}) we introduce a curve $D_t\in \SO(3)$ such that
$$
\vec{p}_t = \sqrt{M}D_t^{-1}e_3.
$$ 
Plugging this expression into (\ref{lax_solution}) we obtain
$$
\sqrt{M}D_t^{-1}e_3 = \sqrt{M}R^{-1}_tD_0^{-1}e_3 \iff D_0R_tD_t^{-1}e_3=e_3.
$$
From this it follows that $D_0R_tD_t^{-1}=e^{\alpha A_3}$ and
\begin{equation}
\label{jurd_sol}
R_t=D_0^{-1}e^{\alpha A_3}D_t.
\end{equation}
If $D_t$ is determined, then after differentiating the last expression, we get a differential equation for $\alpha$, which together with $D_t$ determines a solution $R_t$.

We use Euler angles to find such matrix $D_t$ explicitly:
$$
D_t = e^{\phi_3A_3}e^{\phi_2A_1}e^{\phi_1A_3}.
$$
Since $e^{-\phi_3A_3}e_3 = e_3$, we can assume that $\phi_3 = 0$. Then
$$
\vec{p}_t = \sqrt{M}e^{-\phi_1A_3}e^{-\phi_2A_1}e_3.
$$
Equivalently we have
$$
\sqrt{M}\begin{pmatrix}
\cos\phi_1 & \sin\phi_1 & 0 \\
-\sin\phi_1 & \cos\phi_1 & 0 \\
0 & 0 & 1
\end{pmatrix}
\begin{pmatrix}
1 & 0 & 0 \\
0 & \cos\phi_2 & \sin\phi_2\\
0 & -\sin\phi_2 & \cos\phi_2
\end{pmatrix}
\begin{pmatrix}
0 \\
0\\
1
\end{pmatrix}=
\begin{pmatrix}
p_2 \\
p_1\sqrt{1-a^2} \\
p_3
\end{pmatrix}.
$$
As a result we obtain the following system:
\begin{align*}
\sqrt{M}\sin\phi_1\sin\phi_2 &= p_2, \\
\sqrt{M}\cos\phi_1\sin\phi_2 &= p_1\sqrt{1-a^2}, \\
\sqrt{M}\cos\phi_2 &= p_3;
\end{align*}
from which we get expressions for the components of $D_t$:
\begin{align}
\label{euler_angles}
\cos\phi_2 &= \frac{p_3}{\sqrt{M}},\nonumber \\
\sin\phi_2 &= \sqrt{\frac{M-p_3^2}{M}},\\
\cos\phi_1 &= \frac{p_1\sqrt{1-a^2}}{\sqrt{M-p_3^2}}, \nonumber\\
\sin\phi_1 &= \frac{p_2}{\sqrt{M-p_3^2}}.\nonumber
\end{align}

Finally, we are able to find $\alpha$. We redefine $\alpha = \phi_3$ and consider again expression \eqref{jurd_sol}. Let $R_t'= D_0R_t$. Then
$$
R'_t = e^{\phi_3A_3}D_t = e^{\phi_3A_3}e^{\phi_2A_1}e^{\phi_1A_3}
$$
gives a parameterization of $R'_t$ in terms of Euler angles. Since the whole problem is left-invariant and $D_0$ is constant along each geodesic, we have:
$$
\frac{dR'_t}{dt}=R'_t\left( p_2\sqrt{1-a^2}A_1 + p_1A_2 \right).
$$
After plugging $R'_t = D_0R_t$ in the equations above we get
$$
p_2\sqrt{1-a^2}A_1 + p_1A_2 = (R'_t)^{-1}\dot{R_t}'= \dot{\phi}_3e^{-\phi_1A_3}e^{-\phi_2A_1}A_3e^{\phi_2A_1}e^{\phi_1A_3} + \dot{\phi}_2e^{-\phi_1A_3}A_1e^{\phi_1A_3}+\dot{\phi}_1A_3.
$$
The equivalent vector equality is
$$
p_2\sqrt{1-a^2}e_1 + p_1e_2 = \dot{\phi}_3(\sin\phi_1\sin\phi_2e_1 + \cos\phi_1\sin\phi_2e_2 +\cos\phi_2e_3) + \dot{\phi}_2(\cos\phi_1e_1-\sin\phi_1e_2)+\dot{\phi}_1e_3,
$$
and we get a system of equations
\begin{align*}
\dot{\phi}_3\sin\phi_1\sin\phi_2 + \dot{\phi}_2\cos\phi_1 &= p_2\sqrt{1-a^2}, \\
\dot{\phi}_3\cos\phi_1\sin\phi_2 - \dot{\phi}_2\sin\phi_1 &= p_1, \\
\dot{\phi}_3\cos\phi_2 + \dot{\phi}_1 &= 0.
\end{align*}
From the first two equations we get
$$
\dot{\phi}_3\sin\phi_2 = p_2\sqrt{1-a^2}\sin\phi_1 + p_1\cos\phi_1.
$$
If we use expressions (\ref{euler_angles}) for $\cos\phi_1$, $\sin\phi_1$ and $\sin\phi_2$ we obtain a differential equation for $\phi_3$ :
\begin{equation}
\label{phi_equation}
\pm\sqrt{\frac{M-p_3^2}{M}}\dot{\phi}_3 = \pm\frac{\sqrt{1-a^2}}{\sqrt{M-p_3^2}}\left( p_2^2 + p_1^2 \right) \iff \dot{\phi}_3 = \frac{\sqrt{M(1-a^2)}}{M-p_3^2} = \frac{\sqrt{M(1-a^2)}}{1-a^2p_1^2}.
\end{equation}

We recall that the geodesics will be parameterized in the form $R_t = e^{-\phi_1(0)A_3}e^{-\phi_2(0)A_1}e^{\phi_3A_3}e^{\phi_2A_1}e^{\phi_1A_3}$. 
Then
\begin{align}
\label{R_matrix}
R_t&=\begin{pmatrix}
\frac{p_1(0)\sqrt{1-a^2}}{\sqrt{M-p_3(0)^2}} & \frac{p_2(0)}{\sqrt{M-p_3(0)^2}} & 0 \\
-\frac{p_2(0)}{\sqrt{M-p_3(0)^2}} & \frac{p_1(0)\sqrt{1-a^2}}{\sqrt{M-p_3(0)^2}} & 0 \\
0 & 0 & 1
\end{pmatrix}
\begin{pmatrix}
1 & 0 & 0 \\
0 & \frac{p_3(0)}{\sqrt{M}} & \sqrt{\frac{M-p_3(0)^2}{M}} \\
0 & -\sqrt{\frac{M-p_3(0)^2}{M}} & \frac{p_3(0)}{\sqrt{M}}
\end{pmatrix}
\begin{pmatrix}
\cos\phi_3 & -\sin\phi_3 & 0 \\
\sin\phi_3 & \cos\phi_3 & 0 \\
0 & 0 & 1
\end{pmatrix} \nonumber\\
&\times\begin{pmatrix}
1 & 0 & 0 \\
0 & \frac{p_3}{\sqrt{M}} & -\sqrt{\frac{M-p_3^2}{M}} \\
0 & \sqrt{\frac{M-p_3^2}{M}} & \frac{p_3}{\sqrt{M}}
\end{pmatrix}
\begin{pmatrix}
\frac{p_1\sqrt{1-a^2}}{\sqrt{M-p_3^2}} & -\frac{p_2}{\sqrt{M-p_3^2}} & 0 \\
\frac{p_2}{\sqrt{M-p_3^2}} & \frac{p_1\sqrt{1-a^2}}{\sqrt{M-p_3^2}} & 0 \\
0 & 0 & 1
\end{pmatrix},
\end{align}
where $\phi_3$ satisfies (\ref{phi_equation}). Since $R_0 = \Id$ we have $\phi_3(0)=0$.

Now we compute $M$:

\begin{enumerate}
\item In $C_1$ we have:
\begin{align*}
M=k^2\sn^2 (a\theta) + (1-a^2)\dn^2 (a\theta) + a^2k^2\cn^2 (a\theta) = 
1 - a^2(1-k^2).
\end{align*}

\item  In $C_2$:
$$
M = \sn^2 (a\theta) + (1-a^2)\cn^2 (a\theta) + \frac{a^2}{k^2}\dn^2 (a\theta) = \frac{k^2+a^2(1-k^2)}{k^2}.
$$

\item In $C_3$:
$$
M = \tanh^2 (a\theta) + (1-a^2)\frac{1}{\cosh^2 (a\theta)} + \frac{a^2}{\cosh^2 (a\theta)}= \frac{\sinh^2 (a\theta) + 1}{\cosh^2 (a\theta)} = 1.
$$

\item In $C_4$:
$$
M = 1-a^2.
$$

\item In $C_5$:
$$
M = 1.
$$
\end{enumerate}

We summarize all obtained results in Table~\ref{table_1}.

\begin{table}[h]
\caption{Energy bounds, expressions for $p_i$ and values of $M$ for different regions}
\label{table_1}
\begin{center}
\begin{tabular}{|c|c|c|c|}
\hline
Region & $\begin{array}{c}
\text{Energy} \\
\text{bounds}
\end{array}$ & Expressions for  $p_i$ & Value of $M$ \\
\hline
$C_1$ & $(-a^2,a^2)$ & $\begin{array}{l}
p_1 = s_1\dn (a\theta), \\
p_2 = -s_1 k \sn (a\theta), \\
p_3 = a k \cn (a\theta) \\
\end{array}$ & $1-a^2(1-k^2)$\\
\hline 
$C_2$ & $(a^2,+\infty)$ & $\begin{array}{l}
p_1 = \cn (a\theta/k), \\
p_2 = -s_2 \sn (a\theta/k), \\
p_3 = a s_2 \dn (a\theta/k)/k \\
\end{array}$ & $\dfrac{k^2 + a^2(1-k^2)}{k^2}$ \\
\hline
$C_3$
 & $\{a^2\},p_3\neq 0$ & $\begin{array}{l}
p_1 = s_1/\cosh (a\theta), \\
p_2 = -s_1 s_2 \tanh (a\theta), \\
p_3 = s_2 a/\cosh (a\theta) \\
\end{array}$ & $1$ \\
 \hline
$C_4$
 & $\{-a^2\}$ & $\begin{array}{l}
p_1 = (-1)^n, \\
p_2 = 0, \\
p_3 = 0 \\
\end{array}$ & $1-a^2$ \\
 \hline 
 $C_5$
 & $\{a^2\},p_3= 0$ & $\begin{array}{l}
p_1 = 0, \\
p_2 = (-1)^{n}, \\
p_3 = 0 \\
\end{array}$ & $1$ \\
 \hline 
\end{tabular}
\end{center}
\end{table}

To get a full parameterization of geodesics, we only need to integrate \eqref{phi_equation}. We perform the integration separately for different regions $C_i$. For the definition of elliptic integral of the third kind $\Pi(n;\phi,m)$ see Appendix~B.
\begin{enumerate}
\item Integration in $C_1$:
\begin{align*}
\dot{\phi}_3 &= \frac{\sqrt{M(1-a^2)}}{1-a^2\dn^2(a\theta)} = \frac{\sqrt{M(1-a^2)}}{1-a^2(1-k^2\sn^2(a\theta))}= \sqrt{\frac{M}{1-a^2}}\frac{1}{1+\frac{a^2k^2}{1-a^2}\sn^2(a\theta)}.
\end{align*}
\begin{align*}
\phi_3 &= \sqrt{\frac{M}{1-a^2}}\int_0^t\frac{d\tau}{1+\frac{a^2k^2}{1-a^2}\sn^2(a(\tau + \theta_0))} = \sqrt{\frac{M}{a^2(1-a^2)}}\int_{a\theta_0}^{a\theta}\frac{d\alpha}{1+\frac{a^2k^2}{1-a^2}\sn^2\alpha} \\
 &= \sqrt{\frac{1-a^2(1-k^2)}{a^2(1-a^2)}}\left( \Pi\left( \frac{a^2k^2}{a^2-1}; \am(a\theta) \right) - \Pi\left( \frac{a^2k^2}{a^2-1}; \am(a\theta_0) \right) \right).
\end{align*}

\item Integration in $C_2$:
\begin{align*}
\dot{\phi}_3 &= \frac{\sqrt{M(1-a^2)}}{1-a^2\cn^2\left(\dfrac{a\theta}{k} \right)} =  \frac{\sqrt{M(1-a^2)}}{1-a^2\left(1-\sn^2\left( \dfrac{a\theta}{k} \right) \right)} = \sqrt{\frac{M}{1-a^2}}\frac{1}{1+\frac{a^2}{1-a^2}\sn^2\left( \dfrac{a\theta}{k} \right)}.
\end{align*}
\begin{align*}
\phi_3 &= \sqrt{\frac{M}{1-a^2}}\int_0^t\frac{d\tau}{1+\frac{a^2}{1-a^2}\sn^2\left( \dfrac{a}{k}(\theta_0+\tau) \right)} = k\sqrt{\frac{M}{a^2(1-a^2)}}\int_{\frac{a\theta_0}{k}}^{\frac{a\theta}{k}}\frac{d\alpha}{1+\frac{a^2}{1-a^2}\sn^2\alpha} \\
&= \sqrt{\frac{k^2+a^2(1-k^2)}{a^2(1-a^2)}}\left( \Pi\left( \frac{a^2}{a^2-1}; \am\left(\dfrac{a\theta}{k}\right) \right) - \Pi\left( \frac{a^2}{a^2-1}; \am\left(\dfrac{a\theta_0}{k}\right)\right) \right).
\end{align*}

\item Integration in $C_3$:
\begin{align*}
\dot{\phi}_3 &= \frac{\sqrt{M(1-a^2)}}{1-\frac{a^2}{\cosh^2 (a\theta)}} = \frac{\sqrt{M(1-a^2)}\cosh^2(a\theta)}{\cosh^2a\theta - a^2} = \sqrt{M(1-a^2)}\left(1 - \frac{1}{1-\dfrac{1}{a^2}\cosh^2 (a\theta)}\right).
\end{align*}
We compute the following integral:
\begin{align*}
-\int_0^t\frac{d\tau}{1-\dfrac{1}{a^2}\cosh^2 a(t+\theta_0)}= -\frac{1}{a}\int_{a\theta_0}^{a\theta}\frac{d\alpha}{1-\dfrac{1}{a^2}\cosh^2 \alpha}. 
\end{align*}

After the change of variable $\tanh\alpha = y$ we get $d\alpha = \dfrac{dy}{1-y^2}$ and
\begin{align*}
&-\frac{1}{a}\int_{a\theta_0}^{a\theta}\frac{d\alpha}{1-\dfrac{1}{a^2}\cosh^2 \alpha} \\
&= -\frac{1}{a}\int_{\tanh a \theta_0}^{\tanh a\theta}\frac{dy}{(1-y^2)\left( 1-\frac{1}{a^2(1-y^2)} \right)}=\frac{1}{a}\int_{\tanh a \theta_0}^{\tanh a\theta}\frac{dy}{y^2 + \frac{1-a^2}{a^2}} = \left.\frac{1}{\sqrt{1-a^2}}\arctan\frac{ay}{\sqrt{1-a^2}}\right|_{\tanh a \theta_0}^{\tanh a \theta}.
\end{align*}

As a result we obtain
$$
\phi_3 =\sqrt{1-a^2}t  + \left( \arctan\left( \frac{a}{\sqrt{1-a^2}}\tanh a\theta \right) - \arctan\left( \frac{a}{\sqrt{1-a^2}}\tanh a\theta_0 \right)  \right).
$$

\item Integration in $C_4$:
$$
\dot{\phi}_3 = \frac{\sqrt{M(1-a^2)}}{1-a^2\cos^2(\pi n)} = \frac{\sqrt{M(1-a^2)}}{1-a^2} = 1\, \, \Rightarrow\,\, \phi_3 = t.
$$

\item Integration in $C_5$:
$$
\dot{\phi}_3 = \frac{\sqrt{M(1-a^2)}}{1-a^2\cos^2\left(\frac{\pi}{2}+\pi n\right)} = \sqrt{M(1-a^2)}= \sqrt{1-a^2}\,\, \Rightarrow \,\, \phi_3 = \sqrt{1-a^2}t.
$$
\end{enumerate}

From the parameterization it follows that geodesics which correspond to regions $C_4$ and $C_5$ are uniform rotations around $e_1$ or $e_2$. 

At the end of this section we would like to discuss how to obtain a parameterization of sub-Riemannian geodesics on $S^3$, which is a double cover of SO(3). Consider a family of sub-Riemannian structures $(S^3,\Delta',g')$ where
$$
\Delta' = \spann\{X'_1,X'_2\}, \qquad X'_1 = j/2, \qquad X'_2 =\sqrt{1-a^2}i/2
$$
and
$$
g(X'_l,X'_m) = \delta_{lm}, \qquad l,m=1,2.
$$

Since $X'_l$ satisfy (\ref{structure_equations}), the sub-Riemannian manifolds $(S^3,\Delta',g')$ and $(SO(3),\Delta,g)$ are locally isometric. The parameterization of sub-Riemannian geodesics on $S^3$ can be obtained in the same way as in SO(3). The Hamiltonian system of PMP for the sub-Riemannian problem on $S^3$ is
\begin{align}
\dot{q} &= \frac{q}{2}(p_1j+p_2\sqrt{1-a^2}i),\label{quaternion_system}\\
\dot{p}_1 &= p_2p_3,\nonumber\\
\dot{p}_2 &= -p_1p_3,\nonumber\\
\dot{p}_3 &= a^2p_0p_1;\nonumber
\end{align}
 The horizontal subsystem can be integrated by following the same approach as previously discussed by simply rewriting all expressions in quaternion language using isomorphism between so(3), $\mathbb{R}^3$ and $\I$. As a result we get a parameterization
\begin{equation}
\label{quaternion_param}
q(t) = q_0e^{-\frac{\phi_1(0)}{2}k}e^{-\frac{\phi_2(0)}{2}i}e^{\frac{\phi_3(t)}{2}k}e^{\frac{\phi_2(t)}{2}i}e^{\frac{\phi_1(t)}{2}k}
\end{equation}
where $e^{\frac{\phi_2(t)}{2}i}, e^{\frac{\phi_1(t)}{2}k}$ are quaternions that correspond to rotations $e^{\phi_2(t)A_1}, e^{\phi_1(t)A_3}$ and $\phi_1(t)$, $\phi_2(t)$ and $\phi_3(t)$ are exactly the same as in (\ref{euler_angles}), (\ref{phi_equation}). 

In the next section we will use this parameterization of sub-Riemannian geodesics on SO(3) and $S^3$ to study periodic geodesics on SO(3).

\section{Periodic geodesics on SO(3)}
\label{sec:periodic}

In this section we describe periodic geodesics of the sub-Riemannian problems on SO(3) and study their topological properties.

First we prove the following lemma.
\begin{lemma}
\label{function_lemma}
Consider the following functions:
$$
G_1(a,k) = \sqrt{\dfrac{1-a^2(1-k^2)}{a^2(1-a^2)}}\Pi\left( \dfrac{a^2k^2}{a^2-1}; k^2 \right),
$$
$$
G_2(a,k) = \sqrt{\dfrac{k^2+a^2(1-k^2)}{a^2(1-a^2)}}\Pi\left( \dfrac{a^2}{a^2-1};k^2 \right).
$$
where $\Pi(n;k)$ is the complete elliptic integral of the third kind (see  Appendix B for the definition).

For any fixed $a\in(0,1)$ the functions $G_1(a,k)$ and $G_2(a,k)$ are positive, smooth and increasing on the interval $k\in[0,1)$. Their limit values at $k=0$ and $k=1$ are
$$
\lim_{k\to 0}G_1(a,k) = \dfrac{\pi}{2a}, \qquad \lim_{k\to 1-0}G_1(a,k) = +\infty; 
$$
$$
\lim_{k\to 0}G_2(a,k) = \dfrac{\pi}{2}, \qquad \lim_{k\to 1-0}G_2(a,k) = +\infty;
$$

\end{lemma}

\begin{proof}
The smoothness property follows form the fact that $G_i(a,k)$ is a product of two smooth functions when $k\in[0,1)$. We can differentiate $G_2(a,k)$ with respect to $k$ using formula (\ref{elliptic3_der1}):
$$
\frac{\partial}{\partial k}G_2(a,k) = \frac{kE(k^2)}{a^2(1-k^2)\sqrt{\frac{1}{1-a^2}+\frac{k^2}{a^2}}}.
$$
This expression is non-negative and equal to zero only if $k=0$. Thus the function $G_2(a,k)$ is increasing for $k\in[0,1)$. For $k=0$ we have
$$
G_2(a,0) = \frac{\pi}{2}.
$$
It follows that $G_2(a,k)$ is positive.

Using formulas (\ref{elliptic3_der1}) and (\ref{elliptic3_der2}) we obtain an expression for the derivative of $G_1(a,k)$ with respect to $k$:
\begin{equation}
\label{derivative_fraction}
\frac{\partial G_1}{\partial k} = \frac{E(k^2) - (1-k^2)K(k^2)}{a^2k(1-k^2)\sqrt{\frac{1}{a^2}+\frac{k^2}{1-a^2}}}.
\end{equation}
We want to show that $\partial G_1/\partial k \geq 0, \, k\in(0,1)$.  First differentiate the numerator of fraction (\ref{derivative_fraction}) with respect to $k$:
$$
\frac{\partial}{\partial k}(E(k^2) - (1-k^2)K(k^2)) = kK(k^2) > 0, \quad 0< k < 1.
$$
Since the denominator of (\ref{derivative_fraction}) is positive for any $k\in(0,1)$, and the derivative of the numerator is positive, it is enough to show that the limit of $\partial G_1/\partial k$ when $k\to 0$ is non-negative. Using the asymptotic expansions (\ref{elliptic1_assympt}) and (\ref{elliptic2_assympt}) for $K(k^2)$ and $E(k^2)$  we get
$$
\lim_{k\to 0} \frac{\partial G_1}{\partial k} = \lim_{k\to 0}\frac{\pi(k^2+o(k^2))}{2a^2k(1-k^2)\sqrt{\frac{1}{a^2}+\frac{k^2}{1-a^2}}} = 0.
$$
Therefore $G_1(a,k)$ is increasing at the interval $k\in[0,1)$ for any $a\in(0,1)$. 

For $k=0$ we have
$$
G_1(a,0) = \frac{\pi}{2a}.
$$
Then $G_1(a,k)$ is positive for $a\in(0,1),k\in[0,1)$.

From the definition of the elliptic integral of the third kind it follows that for $a\in(0,1)$
$$
G_1(a,k) \geq \frac{K(k^2)}{1-\frac{a^2}{a^2-1}}, \qquad G_2(a,k) \geq \frac{K(k^2)}{1-\frac{a^2}{a^2-1}}
$$
and from (\ref{elliptic1_limit}) we have, that $K(k^2)\to +\infty$ when $k\to 1-0$. Therefore  $G_1(a,k)\to +\infty$ and $G_2(a,k)\to +\infty$ when $k\to 1-0$.
\end{proof}

\begin{proposition}
For family \textsc{(\ref{R_system})}--\textsc{(\ref{sr_action})} of sub-Riemannian problems on $\SO(3)$ for any value of $a\in(0,1)$ there exists an infinite number of periodic geodesics.
\end{proposition}
\begin{proof}
If a geodesic is periodic, then the covector $p_t$ must be periodic as well with some period~$T$. In the domains $C_1$ and $C_2$ the period $T$ is equal to $4K(k^2)/a$ and $4kK(k^2)/a$ correspondingly. Thus the period of a closed extremal curve must be equal to $mT$,  $m\in\mathbb{N}$. From this it follows that $e^{\phi_1(mT)A_3} = e^{\phi_1(0)A_3}$, $e^{\phi_2(mT)A_1} = e^{\phi_2(0)A_1}$ and from (\ref{R_matrix}) we get $e^{\phi_3(mT)A_3} = \Id$. This is equivalent to $\phi_3(mT) = 2\pi n$. Since $\phi_3 > 0$, we have $n\in\mathbb{N}$.

Now we consider geodesics for which $p_t\in C_1$. From the addition formulas (\ref{elliptic3_sum}) and (\ref{am_sum}) we get

$$
\phi_3(mT) = 2\pi n
\iff \sqrt{\frac{1-a^2(1-k^2)}{a^2(1-a^2)}} \Pi\left( \frac{a^2k^2}{a^2-1}; k^2 \right)  = \frac{\pi}{2}\frac{n}{m}.
$$
Different irreducible fractions $n/m \in \mathbb{Q}_+$ correspond to different periodic geodesics and the existence of an infinite number of  periodic geodesics is reduced to the problem of finding solutions to the equations
\begin{equation}
\label{inf_geodesic}
G_1(a,k) = \frac{\pi}{2}\frac{n}{m}.
\end{equation}
By Lemma \ref{function_lemma}, the function $G_1(a,k)$ is continuous, increasing and for any fixed $a\in(0,1)$ its image is the half-interval $[\pi/(2a),+\infty)$. Therefore for every irreducible fraction $n/m \in \mathbb{Q}_+$ that satisfies
\begin{equation}
\label{condition_1}
\frac{n}{m}>\frac{1}{a}
\end{equation}
there exists a unique solution of (\ref{inf_geodesic}). It is obvious that the number of fractions $n/m \in \mathbb{Q}_+$ that satisfy this condition is infinite, thus the existence of an infinite number of closed geodesics follows.

Using exactly the same argument we prove that there is an infinite number of periodic geodesics such that the corresponding covector $p_t\in C_2$ and the following condition is satisfied:
\begin{equation}
\label{condition_2}
\frac{n}{m}>1.
\end{equation}
If $p_t\in C_2$ then the initial covector $p_0\in C_2$ of the corresponding extremal can be determined from the equation
\begin{equation}
\label{inf_geodesic2}
G_2(a,k)=\frac{\pi}{2}\frac{n}{m}.
\end{equation}

\end{proof}

Apart from the periodic geodesics found in $C_1$ and $C_2$, there are periodic geodesics that correspond to points in $C_4$ and $C_5$. In this case extremal trajectories are just rotations around $e_1$ and $e_2$. There are no other periodic geodesics on SO(3). In fact, for extremal trajectories from $C_3$ the curve $p_t$ is never periodic, and periodic geodesics from $C_1$ and $C_2$ are completely described by conditions (\ref{inf_geodesic}) and (\ref{inf_geodesic2}).

\begin{figure}[h]
\begin{center}
\includegraphics[scale=0.8]{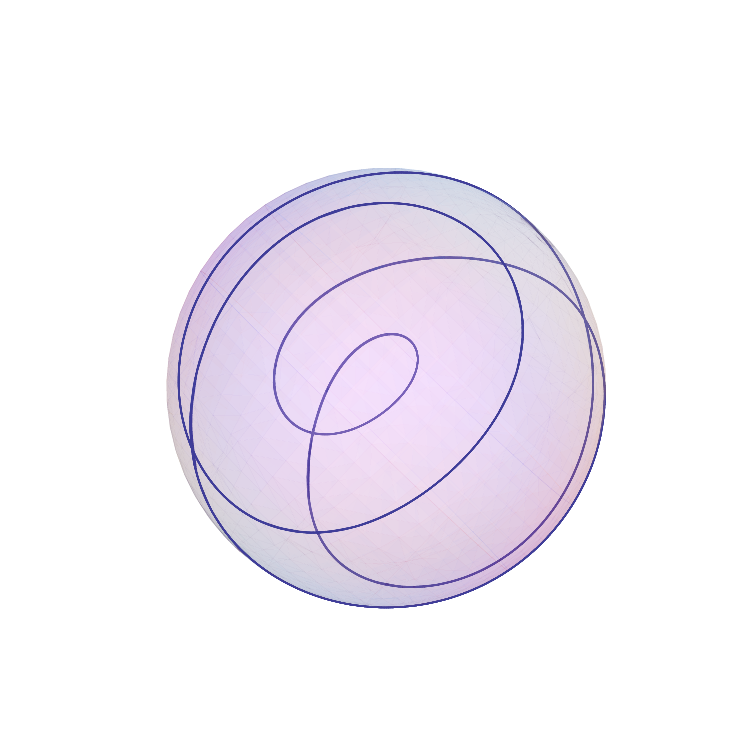}\includegraphics[scale=0.8]{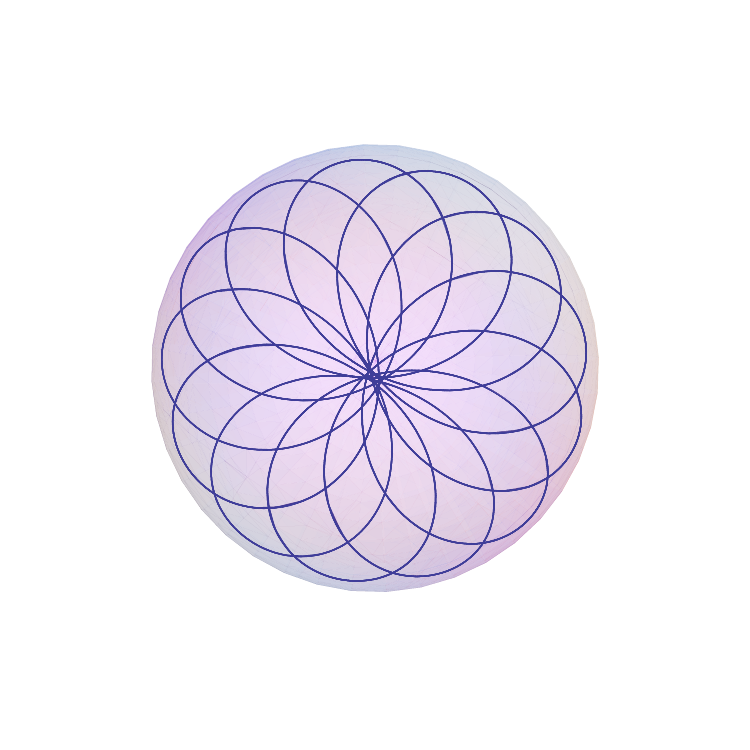}
\caption{Projections of periodic geodesics on $S^2$ with $p_t\in C_1$ and $p_t\in C_2$.}
\end{center}
\end{figure}

Since  $\pi_1(\SO(3))=\mathbb{Z}/2\mathbb{Z}$ there are only two homotopy classes of closed paths in SO(3), and as consequence all non-contractible loops are homotopic. Next we determine which periodic geodesics are null-homotopic. It is well known that a rotation around a fixed vector by $2\pi$ is not contractible in SO(3), but a rotation by $4\pi$ is~\cite{mountain}. Thus it is natural to study contractability of loops on SO(3) when they close up after the first period.

We need two following theorems:

\begin{theorem}
\label{theorem_path}~\textsc{\cite{topology}}
Let $p:X\rightarrow B$ be a covering map and $x_0\in X, b_0 \in B$ be such points that $p(x_0) = b_0$. For any path, i.e., any continuous curve $\gamma: [0,1] \rightarrow B$, that starts at $b_0$, there exists a unique path $\tilde{\gamma}: [0,1] \rightarrow X$ that starts at $x_0$ and such that $\gamma = p\circ\tilde{\gamma}$. The curve $\tilde{\gamma}$ is called the covering path for $\gamma$.
\end{theorem}
\begin{theorem}
\label{theorem_homotopy}~\textsc{\cite{topology2}}
Let $p:X\rightarrow B$ be a covering map where $X$ is the universal cover of $B$. A closed continuous curve $\gamma: [0,1] \rightarrow B$ is null-homotopic if and only if all its covering paths on $X$ are closed.
\end{theorem}

These theorems allow us to find null-homotopic geodesics on SO(3) by studying their covering paths on its universal cover $S^3$.
\begin{proposition}
Consider a periodic geodesic $R_t\in\SO(3)$ that corresponds to a covector curve in $C_1$ or $C_2$, and which is determined by its fraction $n/m\in\mathbb{Q}_+$, satisfying conditions \textsc{(\ref{condition_1})} or \textsc{(\ref{condition_2})}. Then the geodesic $R_t$ is null-homotopic if and only if $n$ is even. All trajectories corresponding to $C_4$ and $C_5$ are non-contractible.
\end{proposition}

\begin{proof}

Lifted sub-Riemannian geodesics of $(SO(3),\Delta,g)$ are exactly the corresponding sub-Riemannian geodesics on $(S^3,\Delta',g')$. From Theorem \ref{theorem_path} it follows that only geodesics on $S^3$ can be covering paths of geodesics from SO(3), and from Theorem~\ref{theorem_homotopy} we know that a geodesic on SO(3) is null-homotopic if and only if all corresponding geodesics on $S^3$ are closed.

Every geodesic on SO(3) has a pair of covering paths on $S^3$ that satisfy $q(0)=\pm 1$ but since the sub-Riemannian structure on $S^3$ is left-invariant, these trajectories belong to the same homotopy class. Thus we can assume that $q(0)=1$. 

Consider first periodic geodesics on SO(3) that correspond to regions $C_1$ and $C_2$ and let us prove that the covering geodesics are periodic only for even $n$. In fact, for periodic geodesics on SO(3) we have $\phi_3(mT) = 2\pi n$. Therefore
$$
e^{\frac{\phi_3(mT)}{2}k} = \cos\frac{\phi_3(mT)}{2} + k \sin\frac{\phi_3(mT)}{2} = \cos \pi n + k \sin \pi n = (-1)^n.
$$
Since $\phi_1(mT)=\phi_1(0)$ and $\phi_2(mT)=\phi_2(0)$, we get from (\ref{quaternion_param})
$$
q(mT) = e^{-\frac{\phi_1(0)}{2}k}e^{-\frac{\phi_2(0)}{2}i}e^{\frac{\phi_3(mT)}{2}k}e^{\frac{\phi_2(mT)}{2}i}e^{\frac{\phi_1(mT)}{2}k} = (-1)^n.
$$
Consequently, if $n$ is even, then the corresponding trajectory on $S^3$ is closed and its projection to SO(3) is null-homotopic. If this is not the case, then the covering geodesic is not closed and its projection is not contractible. 

Since trajectories corresponding to $C_4$ and $C_5$ are just uniform rotations around vectors $e_1$ and $e_2$ they are not contractible.
\end{proof}

Thus we have described all periodic sub-Riemannian geodesics on SO(3) and classified them into two different homotopy classes.

\section{Symmetries of the Hamiltonian system}
\label{sec:symmetries}

A point $Q\in \SO(3)$ is called a Maxwell point for a sub-Riemannian problem on SO(3), if there exist two distinct geodesics of the same length joining $\Id$ with $Q$. 

It is well known that in an analytic sub-Riemannians problem after such a point both geodesics are no longer optimal~\cite{sachkov_se2}. The goal of this section is to obtain some characterization of the Maxwell sets for problem (\ref{R_system})-(\ref{sr_action}). This can be done via a symmetry approach that was successfully applied in~\cite{sachkov_se2}. We begin by looking for some symmetries of the exponential mapping. It is natural to expect that the fixed points of these symmetries are Maxwell points. 

We recall that the exponential mapping $\Exp:C\times\mathbb{R}_+\to\SO(3)$ sends a covector $p\in C = \{p\in\mathfrak{g}^*: H(p) = 1/2\}$ and a instant of time $t$ to the end point of the corresponding geodesic. A pair of mappings $\varepsilon: C\times\mathbb{R}_+\to C\times\mathbb{R}_+$ and $\varepsilon': \SO(3) \to \SO(3)$ is called a symmetry of the exponential map, if the following diagram is commutative:
$$
\begin{CD}
C\times \mathbb{R}_+ @>\Exp>> \SO(3)\\
@V\varepsilon VV @VV\varepsilon'V \\
C\times \mathbb{R}_+ @>\Exp>> \SO(3)
\end{CD}
$$ 

We can construct some symmetries of the exponential map from the symmetries of the Hamiltonian system (\ref{horizontal_subsystem})-(\ref{vertical_subsystem}). We start with the vertical subsystem (\ref{vertical_subsystem}) that has the following symmetries:
\begin{align*}
\varepsilon^1 &: (p_1(s),p_2(s),p_3(s))\mapsto (p_1(t-s),-p_2(t-s),p_3(t-s)), \\
\varepsilon^2 &: (p_1(s),p_2(s),p_3(s))\mapsto (p_1(t-s),p_2(t-s),-p_3(t-s)), \\
\varepsilon^3 &: (p_1(s),p_2(s),p_3(s))\mapsto (p_1(s),-p_2(s),-p_3(s)) ,\\
\varepsilon^4 &: (p_1(s),p_2(s),p_3(s))\mapsto (-p_1(s),-p_2(s),p_3(s)), \\
\varepsilon^5 &: (p_1(s),p_2(s),p_3(s))\mapsto (-p_1(t-s),p_2(t-s),p_3(t-s)), \\
\varepsilon^6 &: (p_1(s),p_2(s),p_3(s))\mapsto (-p_1(t-s),-p_2(t-s),-p_3(t-s)), \\
\varepsilon^7 &: (p_1(s),p_2(s),p_3(s))\mapsto (-p_1(s),p_2(s),-p_3(s)).
\end{align*}

In the phase space of the mathematical pendulum these symmetries are just reflections as it is shown in Figure 2. The variable $\psi$ corresponds to an angle on the $(p_1,p_2)$-plane as it can be seen from (\ref{zamena}).

\begin{figure}
\begin{center}
\includegraphics[scale=0.4]{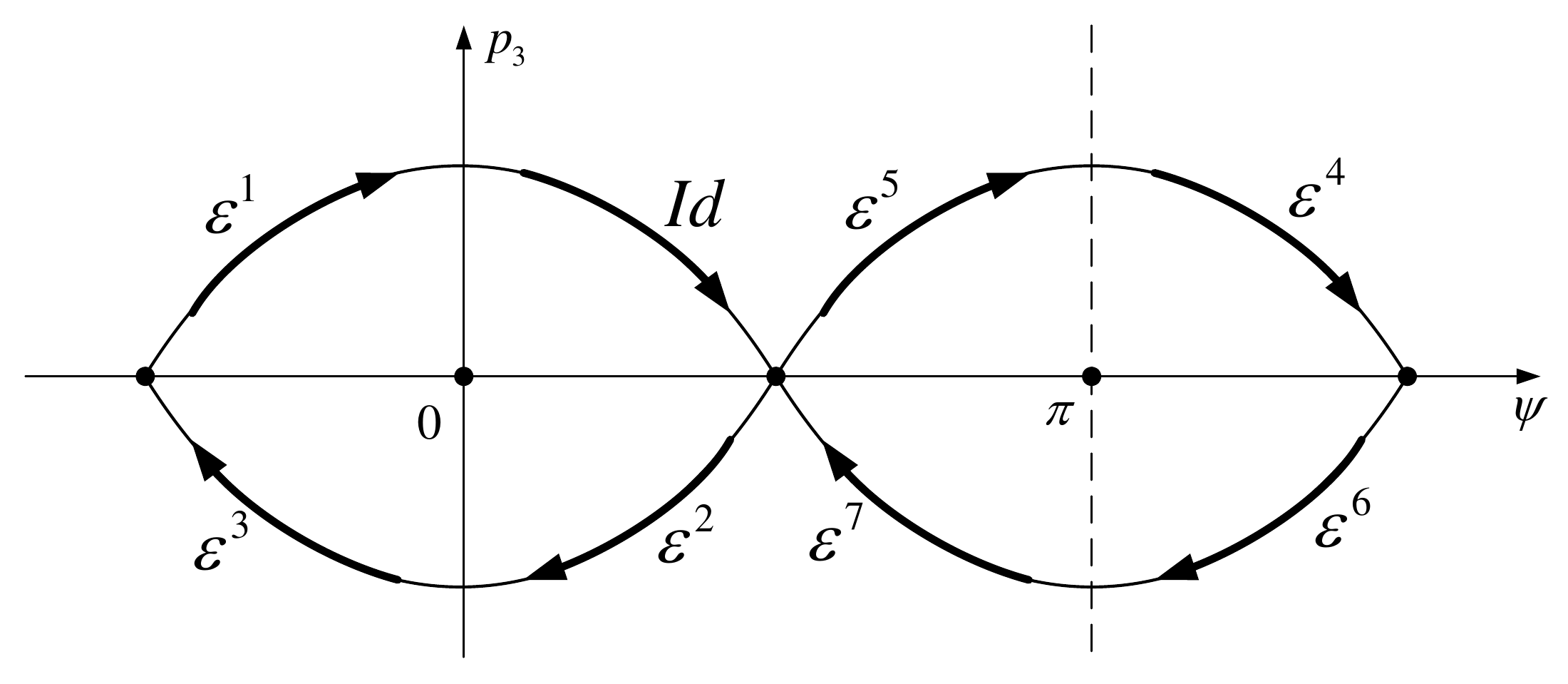}
\end{center}
\caption{Discrete symmetries in the preimage of the exponential map}
\end{figure}

The angular velocity matrix has the form
$$
\Omega_s = \begin{pmatrix}
0 & 0 & p_1(s) \\
0 & 0 & \sqrt{1-a^2}p_2(s) \\
-p_1(s) & -\sqrt{1-a^2}p_2(s) & 0
\end{pmatrix}.
$$
Under the action of $\varepsilon^i$ the components of $\Omega_s = \Omega^1_sA_1 + \Omega^2_sA_2 + \Omega^3_sA_3$ are transformed as follows:
\begin{align}
\varepsilon^1 &: (\Omega^1_s,\Omega^2_s,0)\mapsto (-\Omega^1_{t-s},\Omega^2_{t-s},0), \nonumber\\
\varepsilon^2 &: (\Omega^1_s,\Omega^2_s,0)\mapsto (\Omega^1_{t-s},\Omega^2_{t-s},0), \nonumber\\
\varepsilon^3 &: (\Omega^1_s,\Omega^2_s,0)\mapsto (-\Omega^1_s,\Omega^2_s,0),\nonumber \\
\varepsilon^4 &: (\Omega^1_s,\Omega^2_s,0)\mapsto (-\Omega^1_s,-\Omega^2_s,0), \label{sym_Omega}\\
\varepsilon^5 &: (\Omega^1_s,\Omega^2_s,0)\mapsto (\Omega^1_{t-s},-\Omega^2_{t-s},0), \nonumber\\
\varepsilon^6 &: (\Omega^1_s,\Omega^2_s,0)\mapsto (-\Omega^1_{t-s},-\Omega^2_{t-s},0),\nonumber \\
\varepsilon^7 &: (\Omega^1_s,\Omega^2_s,0)\mapsto (\Omega^1_s,-\Omega^2_s,0) .\nonumber
\end{align}

By using the matrices $I_i = e^{\pi A_i}$, it is easy to check that the action of the symmetries can be written in the following form:
\begin{align*}
\varepsilon^1 &: \Omega_s \mapsto -I_1\Omega_{t-s}I_1, \\
\varepsilon^2 &: \Omega_s \mapsto -I_3\Omega_{t-s}I_3, \\
\varepsilon^3 &: \Omega_s \mapsto I_2\Omega_sI_2, \\
\varepsilon^4 &: \Omega_s \mapsto I_3\Omega_sI_3, \\
\varepsilon^5 &: \Omega_s \mapsto -I_2\Omega_{t-s}I_2, \\
\varepsilon^6 &: \Omega_s \mapsto -\Omega_{t-s}, \\
\varepsilon^7 &: \Omega_s \mapsto I_1\Omega_sI_1.
\end{align*}

Taking into account that $I_i^2 = \Id$ one can show that the mappings defined below are symmetries of the horizontal part of the Hamiltonian system:
\begin{align*}
\varepsilon^1 &: R_s \mapsto I_1R_t^{-1}R_{t-s}I_1, \\
\varepsilon^2 &: R_s \mapsto I_3R_t^{-1}R_{t-s}I_3, \\
\varepsilon^3 &: R_s \mapsto I_2R_sI_2, \\
\varepsilon^4 &: R_s \mapsto I_3R_sI_3, \\
\varepsilon^5 &: R_s \mapsto I_2R_t^{-1}R_{t-s}I_2,\\
\varepsilon^6 &: R_s \mapsto R_t^{-1}R_{t-s}, \\
\varepsilon^7 &: R_s \mapsto I_1R_sI_1.
\end{align*}

The action of $\varepsilon^i$ in the preimage of the exponential map is defined as:
\begin{align*}
\varepsilon^1 &: (t,p_1(0),p_2(0),p_3(0))\mapsto (t,p_1(t),-p_2(t),p_3(t)),\nonumber \\
\varepsilon^2 &: (t,p_1(0),p_2(0),p_3(0))\mapsto (t,p_1(t),p_2(t),-p_3(t)),\nonumber \\
\varepsilon^3 &: (t,p_1(0),p_2(0),p_3(0))\mapsto (t,p_1(0),-p_2(0),-p_3(0)) ,\nonumber\\
\varepsilon^4 &: (t,p_1(0),p_2(0),p_3(0))\mapsto (t,-p_1(0),-p_2(0),p_3(0)),\nonumber \\
\varepsilon^5 &: (t,p_1(0),p_2(0),p_3(0))\mapsto (t,-p_1(t),p_2(t),p_3(t)),\nonumber \\
\varepsilon^6 &: (t,p_1(0),p_2(0),p_3(0))\mapsto (t,-p_1(t),-p_2(t),-p_3(t)),\nonumber \\
\varepsilon^7 &: (t,p_1(0),p_2(0),p_3(0))\mapsto (t,-p_1(0),p_2(0),-p_3(0)).\nonumber
\end{align*}

The action of $\varepsilon^i$ in the image of the exponential map is defined as:
\begin{align*}
\varepsilon^1 &: R_t \mapsto I_1R_t^{-1}I_1, \nonumber\\
\varepsilon^2 &: R_t \mapsto I_3R_t^{-1}I_3, \nonumber\\
\varepsilon^3 &: R_t \mapsto I_2R_tI_2, \nonumber\\
\varepsilon^4 &: R_t \mapsto I_3R_tI_3, \nonumber\\
\varepsilon^5 &: R_t \mapsto I_2R_t^{-1}I_2,\nonumber\\
\varepsilon^6 &: R_t \mapsto R_t^{-1}, \\
\varepsilon^7 &: R_t \mapsto I_1R_tI_1.\nonumber
\end{align*}
Using these definitions it is easy to check that $\varepsilon^i$ are symmetries of the exponential map.

We note that if $\varepsilon^i (p_0) = p_0$ then the corresponding geodesic is mapped to itself. The next proposition gives necessary and sufficient conditions for this to happen.
\begin{proposition}\label{fixed_in_preimage} Let $p_t$ be a solution of \textsc{(\ref{vertical_subsystem})}, $\theta_t$ be the "angle" parameter of the Hamiltonian system for the mathematical pendulum \textsc{(\ref{vertical})} and $a\in(0,1)$. And let $\tau(t,\theta_0)$ and $\xi(t,\theta_0)$ be functions defined as follows:
\begin{align*}
\tau + \xi &= a(t+\theta_0),\\
\tau - \xi &= a\theta_0.
\end{align*}
 Then the following statements are true:
\begin{enumerate}
\item 
$$\varepsilon^1 (p_0) = p_0\iff \begin{cases}
\sn\tau=0, & p_0\in C_1\cup C_2; \\ 
\tau = 0, & p_0\in C_3;
\end{cases}
$$
\item 
$$\varepsilon^2 (p_0) = p_0 \iff \begin{cases}
\cn\tau=0, & p_0\in C_1; \\ 
\text{is impossible for} & p_0\in C_2\cup C_3;
\end{cases}
$$
\item 
$$\varepsilon^5 (p_0) = p_0 \iff \begin{cases}
\cn\tau=0, & p_0\in C_2; \\ 
\text{is impossible for} & p_0\in C_1\cup C_3;
\end{cases}
$$
\item 
$$
\varepsilon^i (p_0) = p_0 \text{ is impossible for } i=3,4,6,7 \text{ and } p_0\in C_1\cup C_2\cup C_3.
$$
\end{enumerate}
\end{proposition} 

\begin{proof}
It is clear from the description of the symmetries that $\varepsilon^i (p_0) = p_0$ is impossible for $i=3,4,7$ and arbitrary $p_0$. The remaining statements are proved very similarly, so we prove just the second one. We have
$$
\varepsilon^2 (p_s) = p_s  \iff \left\{\begin{array}{l}
p_1(t) = p_1(0), \\
p_2(t) = p_2(0), \\
p_3(t) = -p_3(0).
\end{array}\right.
$$
In $C_1$ from the parameterization of the extremals we get an equivalent system:
$$
\left\{\begin{array}{l}
\dn(\tau + \xi) = \dn(\tau - \xi), \\
\sn(\tau + \xi) = \sn(\tau - \xi), \\
\cn(\tau + \xi) = -\cn(\tau - \xi).
\end{array}\right.
$$
Using equations (\ref{sum_sn})-(\ref{sum_dn}) it is easy to see that a solution of this system satisfies $\cn\tau = 0$.

If $p_t \in C_2$ or $p_t \in C_3$ then $\sign (p_3(t)) = \sign(p_3(0))$ for all $t\geq 0$. So it is clear that in this case the equation $p_3(t) = -p_3(0)$ has no solutions.
\end{proof}

Next we prove the main result of this section.
\begin{theorem}
Assume that $R_s\in \SO(3)$, $s\in[0,t]$ is a geodesic and $q_s\in S^3$ is its corresponding quaternion curve. Then $R_s$ is not optimal if for some instant of time $s_0\in(0,t)$ one of the following conditions is satisfied:
\begin{enumerate}
\item $q^0_{s_0} = 0$;
\item $q^1_{s_0} = 0$ and $\sn\tau\neq 0$ if $p_0\in C_1\cup C_2$   or   $\tau \neq 0$ if $p_0\in C_3$;
\item $q^2_{s_0} = 0$ and $\cn \tau\neq 0$  if  $p_0\in C_1$;
\item $q^3_{s_0} = 0$ and $\cn \tau\neq 0$  if  $p_0\in C_2$.

\end{enumerate}
\end{theorem}

\begin{proof}
In view of Proposition~\ref{fixed_in_preimage} and the definition of Maxwell points we only need to show that fixed points of $\varepsilon^i$ in the image of the exponential map satisfy $q^i = 0$.

For the end-point of the geodesic $R_s$ we have 
\begin{align*}
\varepsilon^1 &: R_t \mapsto I_1R_t^{-1}I_1, \\
\varepsilon^2 &: R_t \mapsto I_3R_t^{-1}I_3, \\
\varepsilon^3 &: R_t \mapsto I_2R_tI_2, \\
\varepsilon^4 &: R_t \mapsto I_3R_tI_3, \\
\varepsilon^5 &: R_t \mapsto I_2R_t^{-1}I_2, \\
\varepsilon^6 &: R_t \mapsto R_t^{-1}, \\
\varepsilon^7 &: R_t \mapsto I_1R_tI_1 .
\end{align*}

Consider first the symmetries $\varepsilon^i$ with $i=1,2,5,6$:
$$
\begin{array}{l}
\varepsilon^1 : I_1R_t^{-1}I_1 = R_t, \\
\varepsilon^2 : I_3R_t^{-1}I_3 = R_t,\\
\varepsilon^5 : I_2R_t^{-1}I_2 = R_t,\\
\varepsilon^6 : R_t^{-1} = R_t;
\end{array}
%--------------------
\quad\Rightarrow\quad
%--------------------
\begin{array}{l}
\varepsilon^1 : (R_tI_1)^2 = \Id, \\
\varepsilon^2 : (R_tI_3)^2 = \Id,\\
\varepsilon^5 : (R_tI_2)^2 = \Id,\\
\varepsilon^6 : (R_t)^2 = \Id.
\end{array}
$$
The corresponding quaternion relations are
$$
\begin{array}{l}
\varepsilon^1 : (q_ti)^2 = \pm 1, \\
\varepsilon^2 : (q_tk)^2 = \pm 1,\\
\varepsilon^5 : (q_tj)^2 = \pm 1,\\
\varepsilon^6 : (q_t)^2 = \pm 1,
\end{array}
%--------------------
\quad\Rightarrow\quad
%--------------------
\begin{array}{l}
\varepsilon^1 : 
\left[\begin{array}{l}
Re(q_ti) = 0,\\
q_ti = \pm 1;
\end{array}\right. \\
\varepsilon^2 : 
\left[\begin{array}{l}
Re(q_tk) = 0,\\
q_tk = \pm 1;
\end{array}\right. \\
\varepsilon^5 : 
\left[\begin{array}{l}
Re(q_tj) = 0,\\
q_tj = \pm 1;
\end{array}\right. \\
\varepsilon^7 : 
\left[\begin{array}{l}
Re(q_t) = 0,\\
q_t = \pm 1;
\end{array}\right. \\
\end{array}
%--------------------
\quad\Rightarrow\quad
%--------------------
\begin{array}{l}
\varepsilon^1 : 
\left[\begin{array}{l}
q^1_t = 0,\\
q_t = \pm i;
\end{array}\right. \\
\varepsilon^2 : 
\left[\begin{array}{l}
q^3_t = 0,\\
q_t = \pm k;
\end{array}\right. \\
\varepsilon^5 : 
\left[\begin{array}{l}
q^2_t = 0,\\
q_t = \pm j;
\end{array}\right. \\
\varepsilon^7 : 
\left[\begin{array}{l}
q^t_0 = 0,\\
q_t = \pm 1.
\end{array}\right. \\
\end{array}
$$

For the remaining symmetries we have:
$$
\begin{array}{l}
\varepsilon^3 : I_2R_tI_2 = R_t,\\
\varepsilon^4 : I_3R_tI_3 = R_t,\\
\varepsilon^7 : I_1R_tI_1 = R_t;
\end{array}
%--------------------
\quad\Rightarrow\quad
%--------------------
\begin{array}{l}
\varepsilon^3 : R_tI_2 = I_2R_t,\\
\varepsilon^4 : R_tI_3 = I_3R_t,\\
\varepsilon^7 : R_tI_1 = I_1R_t.
\end{array}
$$
The corresponding quaternion relations are
\begin{align*}
\begin{array}{l}
\varepsilon^3 : q_tj = \pm jq_t,\\
\varepsilon^4 : q_tk = \pm kq_t,\\
\varepsilon^7 : q_ti = \pm iq_t,
\end{array}
%--------------------
&\quad\Rightarrow\quad
%--------------------
\begin{array}{l}
\varepsilon^3 : q^0_tj +q^1_tk-q^2_t-q^3_ti = \pm (q^0_tj -q^1_tk-q^2_t+q^3_ti),\\
\varepsilon^4 : q^0_tk - q^1_tj +q^2_ti -q^3_t = \pm (q^0_tk + q^1_tj - q^2_ti -q^3_t),\\
\varepsilon^7 : q^0_ti - q^1_t - q^2_tk +q^3_tj = \pm (q^0_ti - q^1_t + q^2_tk -q^3_tj);
\end{array}
%--------------------
\quad\Rightarrow\quad
%--------------------
\\
%--------------------
&\quad\Rightarrow\quad
%--------------------
\begin{array}{l}
\varepsilon^3 : 
\left[ \begin{array}{l}
q^1_tk - q^3_ti = 0,\\
q^0_tj - q^2_t = 0;
\end{array}\right.\\
\varepsilon^4 : 
\left[ \begin{array}{l}
q^2_ti - q^1_tj = 0,\\
q^0_tk - q^3_t = 0;
\end{array}\right.\\
\varepsilon^7 :
\left[ \begin{array}{l}
q^3_tj - q^2_tk = 0,\\
q^0_ti - q^1_t = 0.
\end{array}\right.
\end{array}
\end{align*}

All the equations that are different from $q^i_t = 0$ include them as a subsystem. Thus all fixed points of $\varepsilon^i$ in the image of the exponential map satisfy $q^i = 0$.

\end{proof}

We complete this section by discussing the geometric meaning of the symmetries $\varepsilon^i$ in the image of the exponential map. It easy to see that discrete symmetries $\varepsilon^i$ form a finite group $\mathbb{Z}^2\times\mathbb{Z}^2\times\mathbb{Z}^2$. So it is enough to discuss the meaning of some generators of this group, for example, $\varepsilon^3$, $\varepsilon^4$ and $\varepsilon^6$.

Now we look at SO(3) as a unit frame bundle of $S^2$. We can identify an element of SO(3) with a point on the sphere and a tangent vector at this point. If $R\in \SO(3)$, its projection on the sphere is simply given by $R\mapsto Re_1$. 

Let $Re_1 = xe_1 +ye_2 + ze_3$. It is easy to verify by hand that the symmetries $\varepsilon^3$ and $\varepsilon^4$ are just reflections with respect to the plane $y=0$ and the plane $z=0$. These symmetries are shown in Figure~\ref{pic_sym34}. Dashed curves are the reflected curves. 

\begin{figure}[h]
\centering
\subfigure[]{
\includegraphics[width=0.4\linewidth]{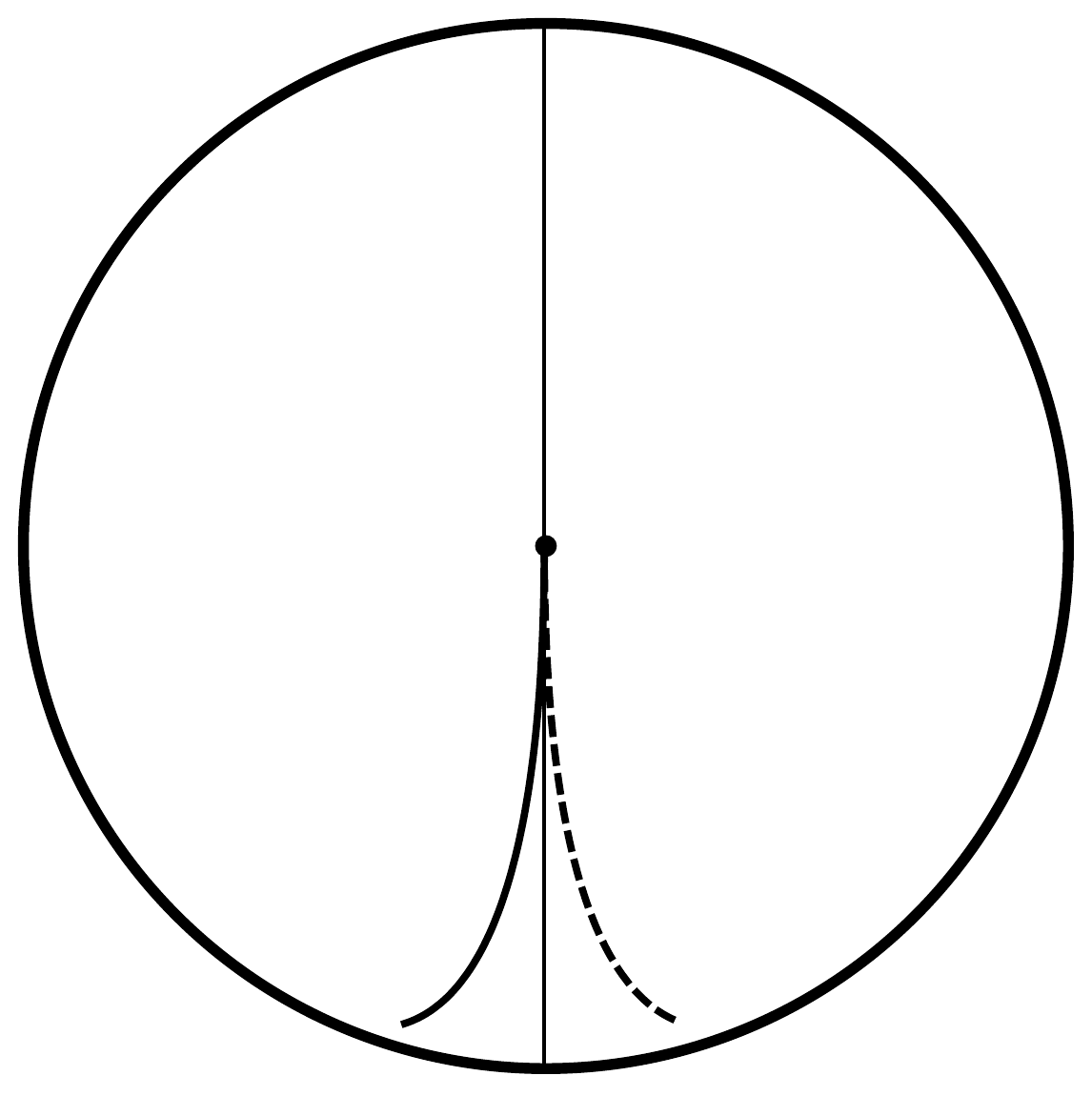}
\label{pic_sym34_a}
}
\hspace{1ex}
\subfigure[]{
\includegraphics[width=0.4\linewidth]{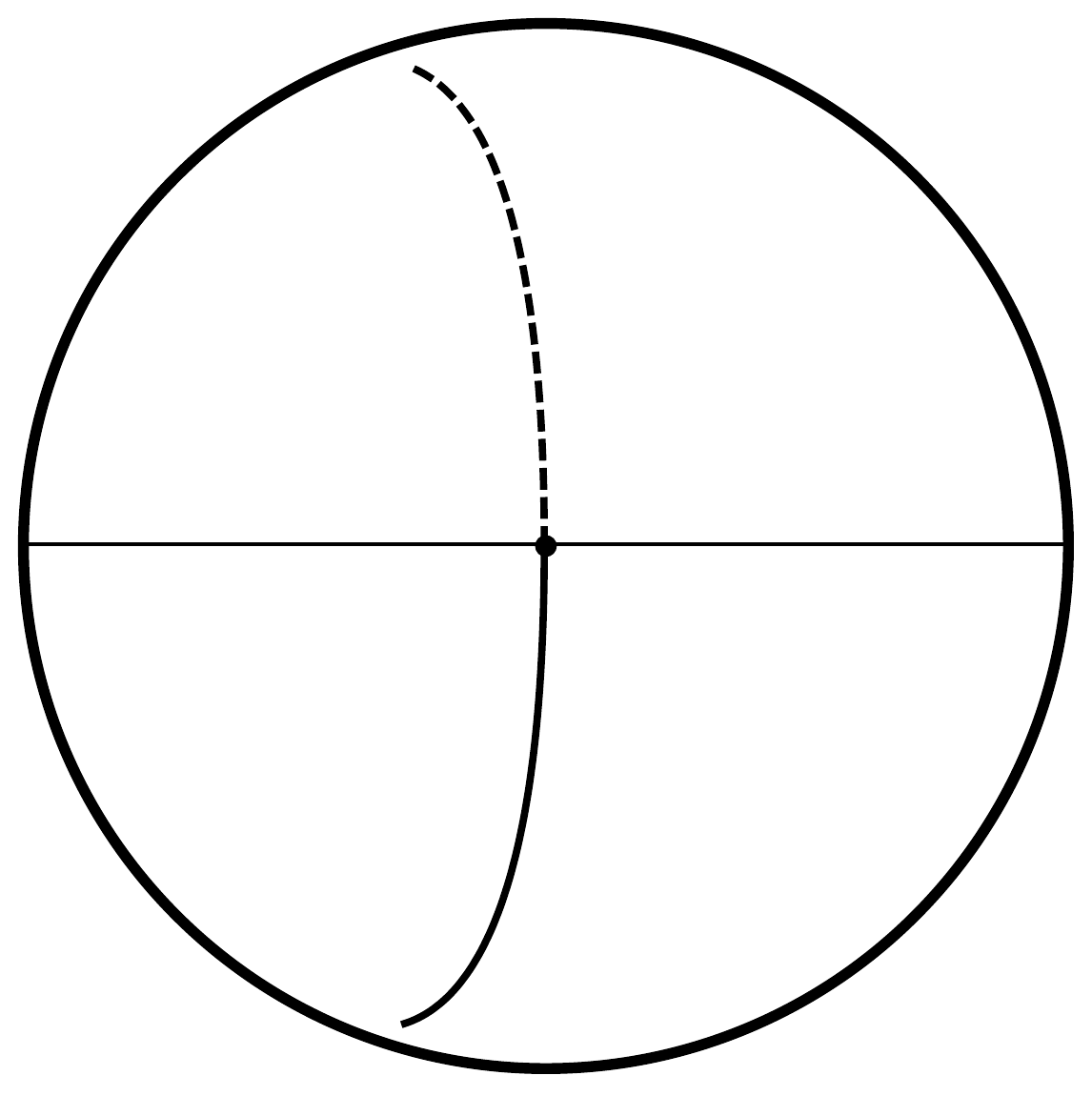}
\label{pic_sym34_b}
}
\caption{Action of discrete symmetries on $S^2$:
\subref{pic_sym34_a} action of $\varepsilon^3$; 
\subref{pic_sym34_b} action of $\varepsilon^4$}
\label{pic_sym34}
\end{figure}

Next we assume that $R_te_1 \neq \pm e_1$. The curve $\varepsilon^6(R_s)e_1$ is up to some rotation a reflection of the curve $R_se_1$ with respect to the center of the chord joining $e_1$ with $R_te_1$. By a chord we mean a short arc of the unique great circle that passes through these two points. 

Write down an analytical expression for this reflection of a curve  $R_se_1$ in terms of quaternions. The centeral point of the chord has coordinates
$$
\vec{c}=\frac{e_1 + R_te_1}{\Vert e_1 + R_te_1\Vert}.
$$
We can rewrite this in quaternion notations:
$$
c=\frac{i + q_tiq_t^{-1}}{\Vert i + q_tiq_t^{-1}\Vert}.
$$
Next we reverse the direction of time on the geodesic $q_s\mapsto q_{t-s}$ and rotate $q_{t-s}iq_{t-s}^{-1}$ around $\vec{c}$ by angle $\pi$. In this way we get an expression for the reflection with respect to the middle point of the considered chord:
$$
a_s = -\frac{(i + q_tiq_t^{-1})q_{t-s}iq_{t-s}^{-1}(i + q_tiq_t^{-1})}{\Vert e_1 + q_te_1q_t^{-1}\Vert^2}.
$$

Now consider an Euler angle parameterization of $R_s$:
\begin{equation*}
R_s = e^{\alpha_3(s)A_1}e^{\alpha_2(s)A_3}e^{\alpha_1(s)A_1}.
\end{equation*}
Note that $\alpha_i(s)$ are different from $\phi_i(s)$ introduced earlier. The claim is that
\begin{equation}
\label{verify}
a_s = e^{(\alpha_3(t)+\alpha_1(t)-\pi)i/2}q_{t}^{-1}q_{t-s}iq_{t-s}^{-1}q_{t}e^{-(\alpha_3(t)+\alpha_1(t)-\pi)i/2}.
\end{equation}
Here $e^{(\alpha_3(t)+\alpha_1(t)-\pi)i/2}$ is a quaternion, that corresponds to a rotation around $e_1$ on angle $\alpha_3(t)+\alpha_1(t)-\pi$. Equation (\ref{verify}) can be verified directly by lengthy computations involving trigonometric functions.

So we have found discrete symmetries of the exponential map and obtained some necessary optimality conditions. 

\begin{figure}[h]
\begin{center}
\includegraphics[scale=0.7]{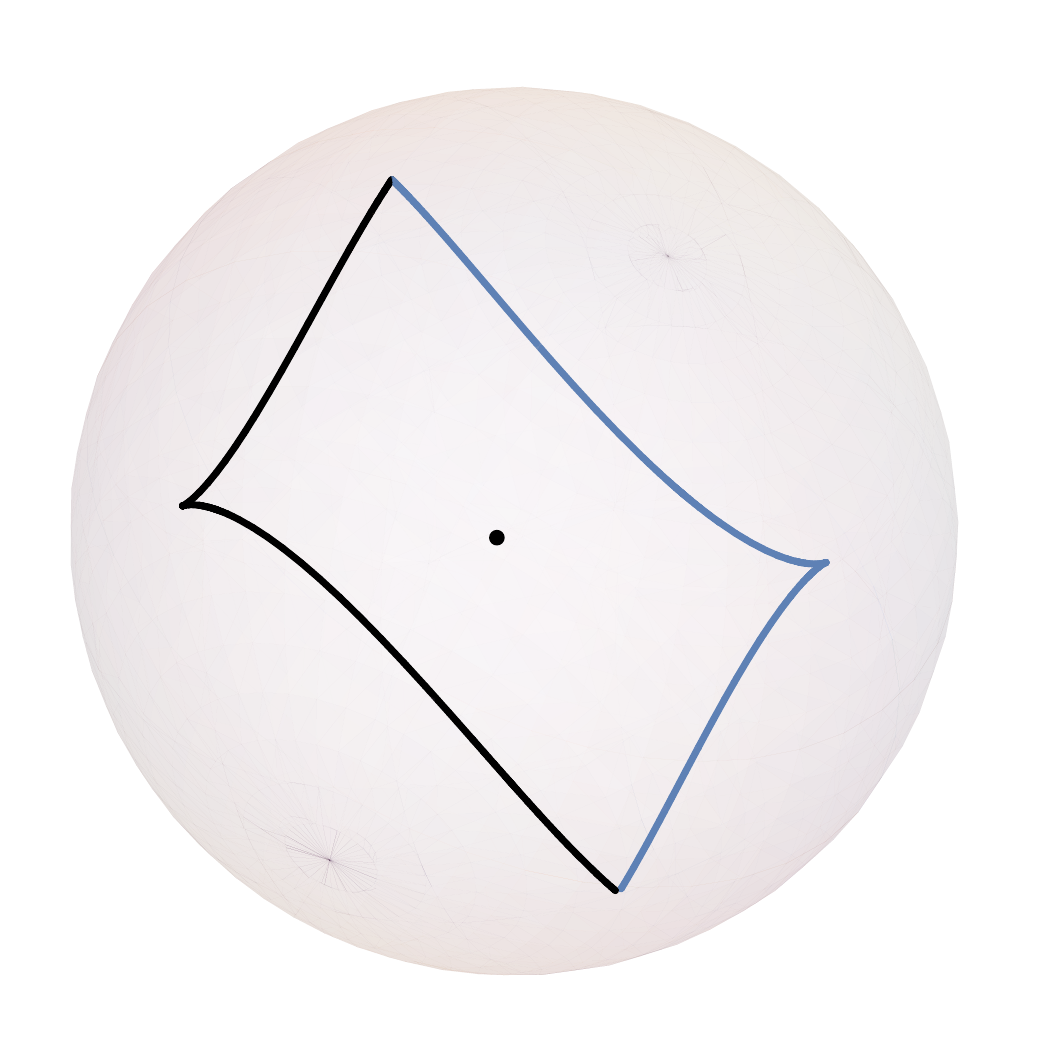}
\end{center}
\caption{Action of $\varepsilon^6$ in the image of the exponential map}
\label{pic_sym62}
\end{figure}

\section{Almost-Riemannian geodesics on $S^2$}
\label{sec:AR-problem}

In the second part of this article we apply the same symmetry approach to study optimality of almost-Riemannian geodesics. We give some new necessary optimality conditions and bounds on the cut time. 

Consider two vector fields on $S^2$ embedded into $\mathbb{R}^3$:
$$
X_1(\vec{\gamma}) = \vec{\gamma}\times e_2, \qquad X_2(\vec{\gamma}) = \sqrt{1-a^2}\vec{\gamma}\times e_1, \qquad \vec{\gamma}\in S^2=\{x\in\mathbb{R}^3\,:\, |x| = 1\},
$$
where $\vec{a}\times\vec{b}$ is the usual cross product between vectors $\vec{a}$ and $\vec{b}$. These vector fields span a rank varying distribution $\vec{\Delta}$ on $S^2$. Assume also that $X_1(\vec{\gamma})$ and $X_2(\vec{\gamma})$ are orthonormal. In this case $S^2$ is endowed with a structure of an almost-Riemannian manifold. Let $\vec{\gamma} = xe_1+ye_2+ze_3$. The set of points where $\rank\vec{\Delta}_{\vec{\gamma}} = 1$ is called the singular set $\mathcal{S}$ and in coordinates $\mathcal{S}=\{\vec{\gamma}\in\mathbb{R}^3: |\vec{\gamma}| = 1, z=0\}$. 

The problem of finding minimal trajectories for the almost-Riemannian structure on $S^2$ can be formulated as an optimal control problem on a sphere:
\begin{equation}
\label{sphere_control}
\dot{\vec{\gamma}} = \vec{\gamma}\times\vec{\omega},
\end{equation}
\begin{equation}
\qquad \vec{\gamma},\vec{\omega}\in \mathbb{R}^3,  \qquad |\vec{\gamma}|=1,\qquad  \vec{\omega} =u_2\sqrt{1-a^2}e_1 + u_1e_2,
\end{equation}
\begin{equation}
\vec{\gamma}(0) = \vec{\gamma}_0, \qquad \vec{\gamma}(T)=\vec{\gamma}_T,
\end{equation}
\begin{equation}
\label{sphere_functional}
\int_0^T \sqrt{u_1^2+u_2^2} dt\rightarrow\min.
\end{equation}
A solution of (\ref{sphere_control}) is
\begin{equation}
\label{sphere_sol}
\vec{\gamma}_t=R^{-1}_t\vec{\gamma}_0,
\end{equation}
where $R_t\in \SO(3)$ satisfies the following differential equation:
$$
\dot{R}=R\Omega, \qquad R(0) = \Id,
$$
and $\Omega \in\so(3)$ is isomorphic to $\vec{\omega}\in\mathbb{R}^3$. The 
matrix $R_t$ is an operator that maps coordinates of a vector in a moving frame to coordinates in the stationary frame.

Optimal control problem (\ref{sphere_control}) can be lifted in a natural way to SO(3):
\begin{equation}
\label{sphere_up_control}
\dot{R} = R\Omega =R(u_2\sqrt{1-a^2}A_1 + u_1A_2),
\end{equation}
\begin{equation}
R\in \SO(3),\qquad\Omega\in \so(3),
\end{equation}
\begin{equation}
R(0) = e^{\beta X_0}, \qquad R(T) = e^{\beta X_0}R_T,
\end{equation}
\begin{equation}
\label{sphere_up_functional}
\int_0^T \sqrt{u_1^2+u_2^2}dt\rightarrow\min.
\end{equation}
where $\beta\in [0,2\pi)$, $X_0\in\so(3)$ is isomorphic to $\vec{\gamma}_0\in\mathbb{R}^3$, $R_T\in\SO(3)$ is an arbitrary special orthogonal matrix, s.t.  $\vec{\gamma}_T=R^{-1}_T\vec{\gamma}_0$, $e^{\beta X_0}$ is the matrix that corresponds to the rotation by angle $\beta$ around the initial vector $\vec{\gamma}_0$. As a result we get an optimal transfer problem between the manifolds $e^{\beta X_0}$ and $e^{\beta X_0}R_T$.

Let $L_{R}: \SO(3) \to \SO(3)$ be the left shift
$$
L_{R}: g\mapsto Rg, \qquad g\in\SO(3),
$$
let $p_t\in\mathfrak{g}^*$ be a solution of (\ref{vertical_subsystem}) and let $\lambda_t\in T^*_{R(t)}\SO(3)$ defined by the relation
$$
\lambda_t = (dL_{R(t)}^*)^{-1}p_t.
$$
Here $dL_{R}:\so(3)\mapsto T_{R}\SO(3) $ is just the differential of the left shift $L_{R}$. Extremal trajectories in problem (\ref{sphere_up_control})-(\ref{sphere_up_functional}) are sub-Riemannian geodesics on $(\SO(3),\Delta,g)$ that satisfy the transversality conditions
$$
\langle \lambda_0,T\left( e^{\beta X_0} \right)\rangle=0, \qquad \langle \lambda_T,T\left(e^{\beta X_0} R_T  \right)\rangle=0.
$$

From left-invariance of the problem it follows that it is sufficient to impose transversality conditions only at the identity element (see~\cite{boscain_quant2} or~\cite{agrachev_sachkov}):
\begin{equation}
\label{transversality_cond}
\langle p_0, X_0 \rangle = 0.
\end{equation}
Using the isomorphism between so(3) and $\mathbb{R}^3$ we can write this in the form
\begin{equation*}
\langle \vec{p}_0, \vec{\gamma}_0 \rangle = -\sin \psi_0x_0+ \cos \psi_0 y_0\sqrt{1-a^2}+p_3(0)z_0  =0.
\end{equation*}
Thus we can use the parameterization of sub-Riemannian geodesics on SO(3) given in Section~\ref{sec:SR-geodesics} to obtain a full paramaterization of almost-Riemannian geodesics on $S^2$. Given an initial point $\vec{\gamma}_0$, any almost-Riemannian geodesic starting from $\vec{\gamma}_0$ is parameterized as $\vec{\gamma}_t = R_t^{-1}\vec{\gamma}_0$, where $R_t$ is a sub-Riemannian geodesic that satisfies the transversality conditions (\ref{transversality_cond}).

\section{Symmetries of the almost-Riemannian problem on $S^2$}
\label{sec:AR-symmetries}

Now we consider symmetries in the otpimal control problem (\ref{sphere_control})-(\ref{sphere_functional}) on the sphere. In the previous section we have seen that almost-Riemannian geodesics on $S^2$ are projections of sub-Riemannian geodesics on SO(3) that satisfy transversality conditions. From this we get a system of equations for almost-Riemannian geodesics
\begin{align}
&\dot{\vec\gamma} = \vec\gamma\times\vec\omega,\label{sphere_hamiltonian}\\
&\dot{\vec{p}} = \vec{p}\times\vec{\omega}\label{sphere_hamiltonian2},\\
&\langle\vec{p}_0,\vec{\gamma}_0\rangle = 0\nonumber
\end{align}
where $\vec\omega = p_1e_2+p_2\sqrt{1-a^2}e_1$. The second equation is just the Lax equation from Section~\ref{sec:SR-geodesics} rewritten in $\mathbb{R}^3$ using the isomorphism between the three-dimensional Euclidean space and so(3) (see Appendix B).

Next we prove the following theorem.
\begin{theorem}
If the initial point of an almost-Riemanian geodesic $\vec{\gamma}_s=x_se_1+y_se_2+z_se_3$ satisfies $x_0 = 0$, $y_0 = 0$ or $z_0 = 0$ and for some instant of time $\tau>0$ we have $x_\tau = 0$, $y_\tau = 0$ or $z_\tau = 0$ correspondingly, then for all $t>\tau$ the geodesic $\vec\gamma_s, s\in[0,t]$ is not optimal.
\end{theorem}

\begin{proof}
Since the vertical subsystem (\ref{sphere_hamiltonian2}) is the same as in the sub-Riemannian case, we consider symmetries  $\varepsilon^i$. From the action of  $\varepsilon^i$ on $\Omega_s$ (see (\ref{sym_Omega})) it follows that the angular velocity vector $\vec{\omega}_s$ is transformed in one of two following ways:
\begin{align*}
\vec{\omega}_s & \mapsto -I_j\vec{\omega}_{t-s},\\
\vec{\omega}_s & \mapsto I_j\vec{\omega}_{s}.
\end{align*}
This allows us to find symmetries of the horizontal part (\ref{sphere_hamiltonian}). It is easy to check that the following mappings are symmetries of system (\ref{sphere_hamiltonian}),(\ref{sphere_hamiltonian2}):
\begin{align}
\varepsilon^1: &\left\{\begin{array}{l}
\vec{p}_s \mapsto -I_2\vec{p}_{t-s},\\ 
\vec\gamma_s \mapsto \pm I_1\vec\gamma_{t-s}, 
\end{array}\right. \label{sym1}\\
\varepsilon^2: &\left\{\begin{array}{l}
\vec{p}_s \mapsto -I_3\vec{p}_{t-s},\\ 
\vec\gamma_s \mapsto \pm I_3\vec\gamma_{t-s}, 
\end{array}\right. \label{sym2}\\
\varepsilon^3:&\left\{\begin{array}{l}
\vec{p}_s \mapsto I_1\vec{p}_{s},\\ 
\vec\gamma_s \mapsto \pm I_2\vec\gamma_{s},
\end{array}\right.& \label{sym3}\\
\varepsilon^4:&\left\{\begin{array}{l}
\vec{p}_s \mapsto I_3\vec{p}_{s},\\ 
\vec\gamma_s \mapsto \pm I_3\vec\gamma_{s},
\end{array}\right.& \label{sym4}\\
\varepsilon^5: &\left\{\begin{array}{l}
\vec{p}_s \mapsto -I_1\vec{p}_{t-s},\\ 
\vec\gamma_s \mapsto \pm I_2\vec\gamma_{t-s}, 
\end{array}\right. \label{sym5}\\
\varepsilon^6: &\left\{\begin{array}{l}
\vec{p}_s \mapsto -\vec{p}_{t-s},\\ 
\vec\gamma_s \mapsto \pm \vec\gamma_{t-s}, 
\end{array}\right. \label{sym6}\\
\varepsilon^7:&\left\{\begin{array}{l} 
\vec{p}_s \mapsto I_2\vec{p}_{s},\\ 
\vec\gamma_s \mapsto \pm I_1\vec\gamma_{s}.
\end{array}\right.& \label{sym7}
\end{align}
We note that each $\varepsilon^i$ represents two symmetries of the Hamiltonian system (\ref{sphere_hamiltonian})-(\ref{sphere_hamiltonian2}), which are characterized by different signs. If these symmetries are also symmetries of the exponential map in an almost-Riemannian problem on $S^2$, then they have to satisfy two extra conditions. First, they must be consistent with the transversality conditions. This is true for all of seven discrete symmetries. In fact, for example, for (\ref{sym3}),(\ref{sym4}),(\ref{sym7}) we have 
$$
\langle I_j\vec{p}_0,\pm I_j\vec{\gamma}_0\rangle = \pm \langle\vec{p}_0,\vec{\gamma}_0\rangle = 0.
$$
Secondly, these symmetries must preserve the initial point of a geodesic. It turns out that the second condition is not always satisfied. 

We write down explicitly the action of the symmetries on $S^2$:
\begin{align*}
\varepsilon^1 &: (x_s,y_s,z_s) \mapsto \pm(-x_{t-s},y_{t-s},z_{t-s}), \\
\varepsilon^2 &: (x_s,y_s,z_s) \mapsto \pm(x_{t-s},y_{t-s},-z_{t-s}), \\
\varepsilon^3 &: (x_s,y_s,z_s) \mapsto \pm(-x_s,y_s,-z_s), \\
\varepsilon^4 &: (x_s,y_s,z_s) \mapsto \pm(-x_s,-y_s,z_s), \\
\varepsilon^5 &: (x_s,y_s,z_s) \mapsto \pm(x_{t-s},-y_{t-s},z_{t-s}), \\
\varepsilon^6 &: (x_s,y_s,z_s) \mapsto \mp(x_{t-s},y_{t-s},z_{t-s}), \\
\varepsilon^7 &: (x_s,y_s,z_s) \mapsto \pm(x_s,-y_s,-z_s).
\end{align*}

Now we find for which sets on $S^2$ the symmetries $\varepsilon^i$, $i=3,4,7$, leave the initial points of extremal trajectories fixed:
\begin{align*}
\begin{array}{l}
\varepsilon^3 : \vec\gamma_0 = \pm I_2\vec\gamma_0,\\
\varepsilon^4 : \vec\gamma_0 = \pm I_3\vec\gamma_0,\\
\varepsilon^7 : \vec\gamma_0 = \pm I_1\vec\gamma_0;
\end{array}
%--------------------
&\quad\Rightarrow\quad
%--------------------
\begin{array}{l}
\varepsilon^3 : x_0e_1+y_0e_2+z_0e_3 = \pm(-x_0e_1+y_0e_2-z_0e_3),\\
\varepsilon^4 : x_0e_1+y_0e_2+z_0e_3 = \pm(-x_0e_1-y_0e_2+z_0e_3),\\
\varepsilon^7 : x_0e_1+y_0e_2+z_0e_3 = \pm(x_0e_1-y_0e_2-z_0e_3);
\end{array}
%--------------------
\quad\Rightarrow\quad
%--------------------
\\
%--------------------
&\quad\Rightarrow\quad
%--------------------
\begin{array}{l}
\varepsilon^3 : 
\left[ \begin{array}{l}
y_0 = 0,\\
y_0 = \pm 1;
\end{array}\right.\\
\varepsilon^4 : 
\left[ \begin{array}{l}
z_0 = 0,\\
z_0 = \pm 1;
\end{array}\right.\\
\varepsilon^7 :
\left[ \begin{array}{l}
x_0 = 0,\\
x_0 = \pm 1.
\end{array}\right.
\end{array}
\end{align*}

To prove the statement of the theorem we construct the symmetries of the exponential map in the almost-Riemannian case similarly to the sub-Riemannian case.
\begin{align*}
\varepsilon^3 &: (t,\vec{p}_0) \mapsto (t,I_1\vec{p}_{0}), & \varepsilon^3 &: \vec\gamma_t \mapsto \pm I_2\vec\gamma_{t}, \\
\varepsilon^4 &:(t,\vec{p}_0) \mapsto (t,I_3\vec{p}_{0}), & \varepsilon^4 &: \vec\gamma_t \mapsto \pm I_3\vec\gamma_{t}, \\
\varepsilon^7 &:(t,\vec{p}_0) \mapsto (t,I_2\vec{p}_{0}); & \varepsilon^7 &: \vec\gamma_t \mapsto \pm I_1\vec\gamma_{t}.
\end{align*}

From Proposition~\ref{fixed_in_preimage} we know that symmetries $\varepsilon^i$, $i=3,4,7$, have no fixed points in the preimage. Any fixed point in the image must satisfy $x=0$, $y=0$ or $z=0$ and from this the proof follows.

A list of symmetries and Maxwell sets is given in Tables~\ref{table_max1} and~\ref{table_max2}. 
\begin{table}[h]
\caption{Symmetries and corresponding Maxwell sets, part I}
\label{table_max1}
\begin{center}
\begin{tabular}{|c|c|c|c|}
\hline
Set & $x_0=0$ & $y_0=0$ & $z_0=0$ \\
\hline
Initial conditions & $p_3 = -\dfrac{p_1y\sqrt{1-a^2}}{z}$ & $p_3 = -\dfrac{p_2x}{z}$ & $p_2x+p_1y\sqrt{1-a^2} = 0$ \\
\hline 
Symmetries & $\varepsilon^7$ & $\varepsilon^3$ & $\varepsilon^4$\\
\hline
Maxwell sets
 & $x_t = 0$ & $y_t = 0$ & $z_t = 0$ \\
 \hline
\end{tabular}
\end{center}
\end{table}
\begin{table}[h]
\caption{Symmetries and corresponding Maxwell sets, part II}
\label{table_max2}
\begin{tabular}{|c|c|c|c|c|c|c|c|c|c|}
\hline
Set & \multicolumn{3}{c|}{$x_0=\pm 1$} & \multicolumn{3}{c|}{$y_0=\pm 1$} & \multicolumn{3}{c|}{$z_0=\pm 1$}\\
\hline
$\begin{array}{c}
\text{Initial}\\
\text{conditions}
\end{array}$ & \multicolumn{3}{c|}{$p_2=0$} & \multicolumn{3}{c|}{$p_1=0$} & \multicolumn{3}{c|}{$p_3=0$}\\
\hline
Symmetries & $\varepsilon^3$ & $\varepsilon^4$ & $\varepsilon^7$ & $\varepsilon^3$ & $\varepsilon^4$ & $\varepsilon^7$ & $\varepsilon^3$ & $\varepsilon^4$ & $\varepsilon^7$\\
\hline
$\begin{array}{c}
\text{Maxwell}\\
\text{sets}
\end{array}$
 &  $y_t = 0$ & $z_t = 0$ & $\begin{array}{l}y_t = 0 \\ z_t = 0\end{array}$ & $\begin{array}{l}x_t = 0 \\ z_t = 0\end{array}$ & $z_t = 0$ & $x_t = 0$ & $y_t = 0$ & $\begin{array}{l}y_t = 0 \\ z_t = 0\end{array}$ & $x_t = 0$\\
 \hline
\end{tabular}
\end{table}

\end{proof}

In articles~\cite{bonnard1,bonnard2} some similar results were obtained.
\begin{theorem}[\cite{bonnard2}]
\label{Bonnard_theorem}
\begin{enumerate}
\item The Gaussian curvature is negative on $S^2/\mathcal{S}$ for all $a\in[0,1)$. 
\item If $\vec\gamma_0\in\mathcal{S}$ then $\mathcal{S}\setminus \{\vec{\gamma}_0\}$ is the cut locus.
\item The geodesic flow has two reflection symmetries: with respect to $\mathcal{S}$ and with respect to the plane $x=0$.
\end{enumerate}
\end{theorem}

First of all, we note that even if the Gaussian curvature is negative everywhere, where it is defined, a geodesic still can have conjugate points if it crosses $\mathcal{S}$~\cite{almost}. That is why this argument allows to find the cut locus only for points on $\mathcal{S}$. Nevertheless any geodesic segment that does not cross the singular set is optimal.

Secondly, since symmetries (\ref{sym1})-(\ref{sym7}) are consistent with the transversality conditions, all seven of them are symmetries of the geodesic flow. But they are symmetries of the exponential map only if they preserve the initial point.

Using formula (\ref{sphere_sol}) one can obtain equations, from which we can find Maxwell times. For all symmetries these expressions have the form
$$
B_s(t)\sin\phi_3(t) + B_c(t) \cos\phi_3(t)  = 0,
$$
where $ B_s(t), B_c(t)$ are coefficients that depend on the initial point $\vec{\gamma}_0$ and $\vec{p}_t$:
\begin{enumerate}
\item $x_0=\pm 1$, equation $y_t=0$:
$$
\sqrt{1-a^2}p_3(t)p_1(t) \sin\phi_3(t) + \sqrt{M}p_2(t) \cos\phi_3(t) = 0;
$$
%----------------------------
\item $x_0=\pm 1$, equation $z_t=0$:
\begin{align*}
\sin\phi_3(t) = 0;
\end{align*}
%----------------------------
\item $y_0=\pm 1$, equation $x_t=0$:
$$
-p_3(t)p_2(t) \sin\phi_3(t) + \sqrt{M(1-a^2)}p_1(t) \cos\phi_3(t) = 0;
$$
%----------------------------
\item $y_0=\pm 1$, equation $z_t=0$:
\begin{align*}
\sin\phi_3(t) = 0;
\end{align*}
%----------------------------
\item $z_0=\pm 1$, equation $x_t=0$:
$$
\sqrt{M(1-a^2)}p_1(t) \sin\phi_3(t) + p_3(t)p_2(t) \cos\phi_3(t) = 0;
$$
%----------------------------
\item $z_0=\pm 1$, equation $y_t=0$:
$$
\sqrt{M} p_2(t) \sin\phi_3(t) -\sqrt{1-a^2}p_3(t)p_1(t) \cos\phi_3(t) = 0;
$$
%----------------------------
\item $z_0=0$, equation $z_t=0$:
\begin{align*}
\sin\phi_3(t) = 0;
\end{align*}
%----------------------------
\item $y_0=0$, equation $y_t=0$:
\begin{align*}
(M z_0 &p_2(t)-(1 - a^2) x_0 p_3(t) p_1(0) p_1(t)) \sin\phi_3(t)  \\
&-\sqrt{1 - a^2} \sqrt{M} \left(z_0 p_3(t) p_1(t) + x_0 p_1(0) p_2(t)\right) \cos\phi_3(t) = 0;
\end{align*}
%----------------------------
\item $x_0=0$, equation $x_t=0$:
\begin{align*}
\sqrt{M} &(\sqrt{1 - a^2} y_0 p_1(t) p_2(0) + z_0 p_3(t) p_2(t)) \sin\phi_3(t)\\
&+\left(y_0 p_3(t) p_2(0) p_2(t)-M \sqrt{1 - a^2} z_0 p_1(t) \right) \cos\phi_3(t) = 0.
\end{align*}
\end{enumerate}
In the expressions above we already used transversality conditions and canceled all non-zero multipliers.

Theorem~\ref{Bonnard_theorem} states that if $\vec{\gamma}_0\in\mathcal{S}$ then the cut time is the first instant of time $t$ when $\vec{\gamma}_t\in\mathcal{S}$. From the equations above it follows that the instant of time $t$ satisfies $\sin \phi_3(t) = 0$. 
\begin{proposition}
\label{sphere_estimates}
The equation $\sin \phi_3(t) = 0$ has positive solutions and the first positive root $t_{0}$ satisfies the following inequalities:
\begin{enumerate}
\item In region $C_1$: $t_{0} \leq 2K(k^2)/a$;
\item In region $C_2$: $t_{0} \leq 2kK(k^2)/a$;
\item In region $C_3$: $t_{0} \leq \pi/\sqrt{1-a^2}$;
\item In region $C_4$: $t_{0} = \pi$;
\item In region $C_5$: $t_{0} = \pi/\sqrt{1-a^2}$.
\end{enumerate}
\end{proposition}

\begin{proof}
The function $\phi_3(t)$ is a monotone increasing function of $t$, which follows from the expression (\ref{phi_equation}) for $\dot{\phi_3}$. Since $\phi_3(0) = 0$, the first positive root of $\sin\phi_3=0$ must satisfy $\phi_3 = \pi$.

By using formula (\ref{elliptic3_sum}) in the region $C_1$ we get:
\begin{align*}
\phi_3\left(\frac{2K(k^2)}{a}\right) &= \sqrt{\frac{1-a^2(1-k^2)}{a^2(1-a^2)}}\left( \Pi\left( \frac{a^2k^2}{a^2-1}; \am(a\theta,k^2)+\pi,k^2 \right) - \Pi\left( \frac{a^2k^2}{a^2-1}; \am(a\theta_0,k^2),k^2 \right) \right) = \\
&= \sqrt{\frac{1-a^2(1-k^2)}{a^2(1-a^2)}}2\Pi\left( \frac{a^2k^2}{a^2-1}; k^2 \right)=2G_1(a,k).
\end{align*}

From Lemma \ref{function_lemma} we get
$$
\phi_3\left(\frac{2K(k^2)}{a}\right)\geq \pi, \quad \forall a\in (0,1),\,k\in[0,1).
$$

Similarly in $C_2$:
\begin{align*}
\phi_3\left(\frac{2kK(k^2)}{a}\right) &= \sqrt{\frac{k^2+a^2(1-k^2)}{a^2(1-a^2)}}2 \Pi\left( \frac{a^2}{a^2-1};k^2 \right) = 2G_2(a,k).
\end{align*}
From Lemma \ref{function_lemma} we get
$$
\phi_3\left(\frac{2kK(k^2)}{a}\right) \geq \pi.
$$

In region $C_3$ we have got earlier this expression
$$
\phi_3 =\sqrt{(1-a^2)}t  + \left( \arctan\left( \frac{a}{\sqrt{1-a^2}}\tanh a\theta \right) - \arctan\left( \frac{a}{\sqrt{1-a^2}}\tanh a\theta_0 \right)  \right).
$$
Since $\tanh a(t+\theta_0)\geq \tanh a\theta_0$  for all $a\in(0,1)$, $\theta_0\in\mathbb{R}$ and $t\geq 0$, the expression in brackets is non-negative. Taking $t=\pi/\sqrt{1-a^2}$ proves the estimate.

The equalities for the regions $C_4$ and $C_5$ are obvious.
\end{proof}
From this we get immediately the following corollary.
\begin{corollary}
For any almost-Riemannian geodesic on $S^2$, s.t. $\vec{\gamma}_0\in\mathcal{S}$ the following bounds on the cut time $t_{cut}$ hold:
\begin{enumerate}
\item In region $C_1$: $t_{cut} \leq 2K(k^2)/a$;
\item In region $C_2$: $t_{cut} \leq 2kK(k^2)/a$;
\item In region $C_3$: $t_{cut} \leq \pi/\sqrt{1-a^2}$;
\item In region $C_4$: $t_{cut} \leq \pi$ ;
\item In region $C_5$: $t_{cut} \leq \pi/\sqrt{1-a^2}$.
\end{enumerate}
\end{corollary}

Using bounds in the almost-Riemannian problem it is possible to give bounds for the cut time in the sub-Riemannian problem on SO(3).

\begin{corollary}
The following bounds on the cut time $t_{cut}$ in the left-invariant sub-Riemannian problems \textsc{(\ref{R_system})-(\ref{sr_action})} on $\SO(3)$ are true:
\begin{enumerate}
\item In region $C_1$: $t_{cut} \leq 2K(k^2)/a + \pi$;
\item In region $C_2$: $t_{cut} \leq 2kK(k^2)/a+ \pi$;
\item In region $C_3$: $t_{cut} \leq \pi/\sqrt{1-a^2} + \pi$;
\item In region $C_4$: $t_{cut} \leq 2\pi$;
\item In region $C_5$: $t_{cut} \leq \pi/\sqrt{1-a^2} + \pi$.
\end{enumerate}
\end{corollary}
\begin{proof}
We construct a trajectory joining $\Id$ with $R_T\in \SO(3)$ that consists of two geodesic segments. Since the sub-Riemannian problem on SO(3) is left-invariant, we can look for a curve that connects $R_T^{-1}$ with $\Id$. 

First we find an almost-Riemannian geodesic $\gamma_s:[0,\tau]\to S^2$ that connects $R_Te_2$ with $e_2$. From Theorem~\ref{Bonnard_theorem} we know that this geodesic arc is minimal. Its length can be estimated by Proposition~\ref{sphere_estimates}.

Next we take the corresponding sub-Riemannian geodesic $R_s:[0,\tau]\to SO(3)$ that has the same vertical curve $p_t$. The terminal position $R_\tau$ is up to a rotation around the vector $e_2$ the identity element. But we know that a rotation around $e_2$ is a sub-Riemannian geodesic on SO(3) of length at most $\pi$. From this the proof follows. 
\end{proof}

\section*{Conclusion}
\label{conclusion}
\addcontentsline{toc}{section}{Conclusion}

In this article we have studied the left-invariant sub-Riemannian problem on SO(3) and the almost-Riemannian problem on $S^2$ which are connected with each other. In both problems we have studied the symmetries of the exponential map and obtained some necessary optimality conditions. We gave a description of periodic sub-Riemannian geodesics and studied some of their topological properties. Finally we have obtained some bounds on the cut time in the almost-Riemannian problem and constructed from this estimates on the cut time in the sub-Riemannian problem.

In the future we plan to obtain bounds on the Maxwell time and the conjugate time in sub-Riemannian problems on SO(3). This might allow us to construct optimal synthesis similar to the case of SE(2)~\cite{sachkov_se2, sachkov2,sachkov3}. In the almost-Riemannian problem it would be interesting to study some characteristics of the cut and conjugate locus for general points on $S^2$. Numerical experiments show that for a generic initial point the conjugate locus has four cusps, similar to the case of Riemannian problem on an ellipsoid. It would be interesting to give a rigorous prove of this statement and also to find how many cusps has a conjugate locus of a point on the singular set. At last it would be interesting to obtain a full optimal synthesis for the points $z_0 =\pm 1$. Numerical simulations suggest that the conjugate locus in this case is a symmetric astroid and the cut locus is a segment that connects its two opposite cusps. We note that the Gauss curvature argument will not work here, since the initial point does not lie on the singular set. We have tried to use the comparison theorems approach from~\cite{boscain_quant} combined with the results from this paper, but we were able only to show the absence of cut points before the singular set. Solving this particular problem may help to obtain optimal synthesis for points that lie outside the singular set and give ideas how to deal with singular sets in other almost-Riemannian problems.

\subsection*{Acknowledgements.}
The authors thank A.A. Agrachev for useful suggestions and comments.

\appendix
\section{Quaternions, SO(3) and $\mathbb{R}^3$}

Let $\mathbb{H} = \{q=q_0+iq_1+jq_2+kq_3 : q_0,...,q_3 \in\mathbb{R}\}$ be the quaternion algebra. The length of $q \in \mathbb{H}$ is defined as $|q| = \sqrt{q_0^2 +q_1^2 + q_2^2 + q_3^2}$. The quaternion $\overline{q} = q_0 - iq_1 - jq_2 - kq_3$ is called conjugate to $q$. The inverse quaternion of $q\neq 0$ is
$$
q^{-1} = \frac{\overline{q}}{|q|^2}.
$$

Let $S^3 = \{q\in \mathbb{H} :|q| = 1\}$ be a three-dimensional unit sphere and $I =\{q\in \mathbb{H}:q_0 = 0 \}$ be the space of imaginary quaternions, which is naturally identified with $\mathbb{R}^3$. Every quaternion $q \in S^3$ defines a rotation operator $R_q$ of any vector $a \in I$ via conjugation:
$$
R_q: a \mapsto qaq^{-1}\in I.
$$
For every $R_q$ there are two distinct quaternions $q$ and $-q$ in $S^3$ that correspond to the same rotation operator and therefore the mapping $p: q \mapsto R_q$ gives a double cover of $S^3$ over $\SO(3)$. This covering is a homeomorphism~\cite{quaternion} which is given in coordinates by
\begin{equation}
\label{projection_map}
p:q\mapsto\begin{pmatrix}
q_0^2+q_1^2-q_2^2-q_3^2 & 2(q_1q_2-q_0q_3) & 2(q_0q_2+q_1q_3) \\
2(q_1q_2+q_0q_3) & q_0^2-q_1^2+q_2^2-q_3^2 & 2(q_2q_3-q_0q_1) \\
2(q_1q_3-q_0q_2) & 2(q_2q_3+q_0q_1) & q_0^2-q_1^2-q_2^2+q_3^2
\end{pmatrix}.
\end{equation}

 \sloppy 
If $R\in \SO(3)$ is the rotation operator around a nonzero vector $\vec{a} = (a_1, a_2, a_3) \in \mathbb{R}^3$ by an angle $\beta$ then the corresponding unit quaternion $q\in S^3$ has the form:
\fussy
\begin{equation}
\label{quaternion_rule}
q = \cos\frac{\beta}{2}+\frac{a_1i+a_2j+a_3k}{|\vec{a}|}\sin\frac{\beta}{2}.
\end{equation}

The space I of imaginary quaternions is a Lie algebra with a Lie bracket
$$
[a,b]=\frac{ab-ba}{2}.
$$
All three spaces $I$, so(3) and $\mathbb{R}^3$ with the cross product are isomorphic as Lie algebras. This isomorphism is given by:
\begin{equation}
\label{algebra_iso}
A = a_1A_1 + a_2A_2 + a_3A_3 \simeq a = a_1i + a_2j + a_3j \simeq \vec{a}=a_1e_1 + a_2e_2 + a_3e_3.
\end{equation}

Suppose that $R\in\SO(3)$ and $q\in p^{-1}(R)$. Then the following is true for any $A\in \so(3)\simeq a\in I  \simeq \vec{a}\in\mathbb{R}^3$:
\begin{equation}
\label{iso_property}
RAR^{-1} \simeq qaq^{-1}\simeq R\vec{a}.
\end{equation}

The Lie algebras I, SO(3) and $\mathbb{R}^3$ carry a natural scalar product given by
$$
(a,b)=(\vec{a},\vec{b})=-\frac{1}{2}\tr(AB) =a_1b_1+a_2b_2+a_3b_3.
$$
This is just the Killing form used in Section~\ref{sec:SR-geodesics}.

\section{Elliptic integrals and elliptic functions}

In this article we have used the following definitions.
\begin{enumerate}
\item Elliptic integral of the first kind:
\begin{equation}
\label{elliptic_1}
F(\phi;k^2)=\int_0^{\phi}\frac{d\theta}{\sqrt{1-k^2\sin^2\theta}};
\end{equation}
\item Elliptic integral of the second kind:
\begin{equation}
\label{elliptic_2}
E(\phi;k^2)=\int_0^{\phi}\sqrt{1-k^2\sin^2\theta}d\theta
\end{equation}
\end{enumerate}
where $0\leq k<1$. The complete elliptic integrals are defined as $K(k^2)=F(\pi/2;k^2)$ and $E(k^2)=E(\pi/2;k^2)$.

The Jacobi amplitude function $\am(\theta;k^2)$ is the inverse of the elliptic integral of the first kind with respect to $\theta$. The Jacobi elliptic functions are defined in the following way:
\begin{align*}
\sn(\theta;k^2)&=\sin \left( \am(\theta;k^2) \right);\\
\cn(\theta;k^2)&=\cos \left( \am(\theta;k^2) \right);\\
\dn(\theta;k^2)&=\sqrt{1-k^2\sn^2(\theta;k^2)}.
\end{align*}
The functions  $\sn(\theta;k^2)$ and $\cn(\theta;k^2)$ are $4K(k^2)$-periodic and $\dn(\theta;k^2)$ is $2K(k^2)$-periodic. When it does not lead to any confusion we omit the modulus $k^2$.

For the Jacobi elliptic functions we have the addition formulas~\cite{byrd}
\begin{equation}
\label{sum_sn}
\sn(a\pm b)=\frac{\sn a\cn b\dn b \pm \sn b\cn a\dn a}{1-k^2 \sn^2 a \sn^2 b},
\end{equation}
\begin{equation}
\label{sum_cn}
\cn(a\pm b)=\frac{\cn a\cn b\mp \sn a\sn b\dn a\dn b}{1-k^2 \sn^2 a \sn^2 b},
\end{equation}
\begin{equation}
\label{sum_dn}
\dn(a\pm b)=\frac{\dn a\dn b\mp \sn a\sn b\cn a\cn b}{1-k^2 \sn^2 a \sn^2 b}.
\end{equation}

Elliptic integral of the third kind:
$$
\Pi(n,\phi;k^2)=\int_0^{\phi}\frac{d\theta}{(1-n\sin^2\theta)\sqrt{1-k^2\sin^2\theta}}.
$$
With a change of variables $\sin\theta = \sn(\alpha,k^2)$ it takes the form
\begin{equation}
\label{elliptic_3}
\Pi(n,\phi;k^2)=\int_0^{F(\phi;k^2)}\frac{d\alpha}{1-n\sn^2(\alpha,k^2)}
\end{equation}
and the complete elliptic integral of the third kind is denoted by $\Pi(n;k^2)=\Pi(n,\pi/2;k^2)$.

All three elliptic integrals satisfy a simple addition property of the form~\cite{lawden}
\begin{equation}
\label{elliptic1_sum}
F(\phi+m\pi;k^2) = F(\phi;k^2) + 2mK(k^2),
\end{equation}
\begin{equation}
\label{elliptic2_sum}
E(\phi+m\pi;k^2) = E(\phi;k^2) + 2mE(k^2),
\end{equation}
\begin{equation}
\label{elliptic3_sum}
\Pi(n,\phi+m\pi;k^2) = \Pi(n,\phi;k^2) + 2m\Pi(n;k^2).
\end{equation}

From (\ref{elliptic1_sum}) one can derive an analogous formula for $\am(\theta,k^2)$:
\begin{equation}
\label{am_sum}
\am(\theta + 2mK(k^2);k^2) = \am(\theta;k^2) + m\pi.
\end{equation}

The following formulas for the derivatives of the elliptic integrals are true~\cite{byrd}:
\begin{equation}
\label{elliptic1_der}
\frac{d K(k^2)}{dk} = \frac{E(k^2)-(1-k^2)K(k^2)}{k(1-k^2)},
\end{equation}
\begin{equation}
\label{elliptic2_der}
\frac{d E(k^2)}{dk} = \frac{E(k^2)-K(k^2)}{k},
\end{equation}
\begin{equation}
\label{elliptic3_der1}
\frac{\partial \Pi(n;k^2)}{\partial k} = \frac{k}{(1-k^2)(k^2-n)}\left( E(k^2)-(1-k^2)\Pi(n;k^2) \right),
\end{equation}
\begin{equation}
\label{elliptic3_der2}
\frac{\partial \Pi(n;k^2)}{\partial n} = \frac{\Pi(n;k^2)-K(k^2)}{2n(1-n)}-\frac{E(k^2)}{2(k^2-n)(1-n)}.
\end{equation}

The following formulas for the asymptotic expansions of the elliptic integrals when $k\to 0$ are true~\cite{byrd}:
\begin{equation}
\label{elliptic1_assympt}
K(k^2) = \frac{\pi}{2}\left( 1 + \frac{k^2}{4} \right) + O(k^4),
\end{equation}
\begin{equation}
\label{elliptic2_assympt}
E(k^2) = \frac{\pi}{2}\left( 1 - \frac{k^2}{4} \right) + O(k^4).
\end{equation}

When $k\to 1-0$ we have~\cite{byrd}:
\begin{equation}
\label{elliptic1_limit}
\lim_{k\to 1-0} \left( K(k^2) - \ln\frac{4}{\sqrt{1-k^2}} \right)=0.
\end{equation}
%%-----------------------------
%%      your bibliography
%%-----------------------------

\end{document}